\documentclass{amsart}
\usepackage{amsmath,amsthm,amsfonts,amssymb,enumitem}

\newtheorem{thm}{Theorem}[section]
\newtheorem{cor}[thm]{Corollary}
\newtheorem{prop}[thm]{Proposition}
\newtheorem{lem}[thm]{Lemma}

\theoremstyle{definition}
\newtheorem{defn}[thm]{Definition}
\newtheorem{exmp}[thm]{Example}

\newcommand\ncrank{\operatorname{nc-rank}}

\title[Linear and matrix min-max theorems]{Linear and matrix generalizations of some combinatorial min-max theorems}
\author{Nik Weaver}
\date{May 2026}

\begin{document}

\keywords{combinatorial min-max theorems, noncommutative rank, nilpotent operator algebras, Hall's marriage theorem, Menger's theorem, coherent paths}
\subjclass{primary 15A03, 05C50; secondary 15A30, 05D15, 05C40}

\begin{abstract}
We review known linear and matrix generalizations of Hall's classic ``marriage theorem'' and Kőnig's theorem on partial matchings in bipartite graphs, and relate them to linear and matrix generalizations of Dilworth's theorem about chains and antichains in posets and Menger's theorem about disjoint paths in directed graphs.
\end{abstract}

\maketitle

\section{The linear setting}

\subsection{Hall's theorem}

Hall's marriage theorem \cite{hall}, in its most memorable (albeit informal, and perhaps outdated) formulation \cite{HV}, says that given $n$ boys and $m$ girls, if for every $1 \leq k \leq n$ every set of $k$ boys collectively know at least $k$ girls, then the boys can be married off in such a way that each boy marries a girl he knows.

This theorem can be generalized to the linear setting. Let the scalar field be $\mathbb{F} = \mathbb{R}$ or $\mathbb{C}$. Given a relation $R \subseteq \mathbb{F}^n \times \mathbb{F}^m$ with $n \leq m$, define the {\it neighborhood} of a nonzero vector $u \in \mathbb{F}^n$ relative to $R$ to be the set $$N(u) = \{w \in \mathbb{F}^m:\mbox{ there exists }v \in \mathbb{F}^n\mbox{ with }(v,w) \in R\mbox{ and }u \not\perp v\},$$ and let the neighborhood of any set $S \subseteq \mathbb{F}^n$ be $$N(S) = \bigcup_{u \in S} N(u) \subseteq \mathbb{F}^m.$$ We define a {\it saturated matching} in $R$ to be a set of pairs $(v_1,w_1), \ldots, (v_n,w_n) \in R$ such that the $v_i$'s constitute a basis of $\mathbb{F}^n$ and the $w_i$'s are linearly independent in $\mathbb{F}^m$.

\begin{thm}\label{qhmt} (Linear marriage theorem)
Let $m \geq n \geq 1$ and let $R \subset \mathbb{F}^n \times \mathbb{F}^m$ be finite. Then there is a saturated matching in $R$ if and only if $${\rm dim}\, {\rm span}(N(S)) \geq {\rm dim}\, {\rm span}(S)$$ for every finite subset $S$ of $\mathbb{F}^n$.
\end{thm}

For the forward direction, suppose $(v_1, w_1), \ldots, (v_n,w_n)$ is a saturated matching in $R$. For any finite $S \subset \mathbb{F}^n$, if ${\rm dim}\, {\rm span}(S) = k$ then the orthocomplement of ${\rm span}(S)$ has dimension $n - k$ and hence can contain at most $n - k$ of the $v_i$'s. This means that at least $k$ of the $v$'s contribute to $N(S)$, i.e., $N(S)$ contains at least $k$ of the $w_i$'s and hence its span has dimension at least $k$. Thus ${\rm dim}\, {\rm span}(N(S)) \geq {\rm dim}\, {\rm span}(S)$ for every $S$. The converse is easily deduced from Edmond's matroid intersection theorem \cite[Theorem (69)]{edmonds} (cf.\ \cite[Theorem 1*]{lovasz}), but it is amusing to note that one can also give an elementary proof that follows the pattern of Halmos and Vaughan's inductive proof of the classical marriage theorem \cite{HV}: assume the theorem is true for all smaller values of $n$ and, for the given value of $n$, for all smaller $|R|$ (= the cardinality of $R$), then separate into the two cases where (1) for every $1 \leq k < n$, every linearly independent set $S \subset \mathbb{F}^n$ of cardinality $k$ satisfies ${\rm dim}\, {\rm span}(N(S)) \geq k + 1$ and (2) for some $1 \leq k < n$ there is a linearly independent set $S \subset \mathbb{F}^n$ of cardinality $k$ with ${\rm dim}\, {\rm span}(N(S)) = k$ exactly. We can then invoke the induction hypothesis, in the first case for smaller $|R|$ and in the second case for smaller $n$. The interested reader is invited to fill in the missing details.

In the special case where $n = m$, we get a {\it perfect matching} in $R$, where the $v_i$'s and the $w_i$'s both form bases of $\mathbb{F}^n$ and the map $v_i \mapsto w_i$ extends to a linear isomorphism of $\mathbb{F}^n$ with itself, rather than merely a linear injection from $\mathbb{F}^n$ into $\mathbb{F}^m$.

We can draw several consequences from Theorem \ref{qhmt}. First, the classical marriage theorem follows easily from it. Given a relation $r \subseteq \{1, \ldots, n\} \times \{1, \ldots, m\}$, let the {\it neighborhood} of a set $s \subseteq \{1, \ldots, n\}$ relative to $r$ be the set $n(s) = \{j \in \{1, \ldots, m\}$: there exists $i \in s$ with $(i,j) \in r\}$. Let a {\it saturated matching} in $r$ be a set of pairs $(i_1, j_1), \ldots, (i_n, j_n) \in r$ such that $\{i_1, \ldots, i_n\} = \{1, \ldots, n\}$ and the $j$'s are distinct.

\begin{cor} \label{qhmtc} (Hall's theorem)
Let $m \geq n \geq 1$ and let $r \subseteq \{1, \ldots, n\}\times \{1, \ldots, m\}$. Then there is a saturated matching in $r$ if and only if $|n(s)| \geq |s|$ for all $s \subseteq \{1, \ldots, n\}$.
\end{cor}

\begin{proof}
The forward implication is trivial. For the reverse implication, assume $|n(s)| \geq |s|$ for all $s \subseteq \{1, \ldots, n\}$ and let $R \subset \mathbb{F}^n \times \mathbb{F}^m$ be the set of pairs $(e_i, e_j)$ with $(i,j) \in r$, where $(e_i)$ are the standard basis vectors. For any finite subset $S$ of $\mathbb{F}^n$, let $s = \{i \in \{1, \ldots, n\}: e_i \not\perp S\}$. Then $S \subseteq {\rm span}\{e_i: i \in s\}$, so ${\rm dim}\, {\rm span}(S) \leq |s|$, and $N(S) = \{e_j: j \in n(s)\}$, so ${\rm dim}\, {\rm span}(N(S)) = |n(s)|$. Thus the cardinality condition on $s$ assumed here entails the dimension condition on $S$ from Theorem \ref{qhmt}, and we conclude from that result that there is a saturated matching in $R$ in the sense used there. This immediately implies that there is a saturated matching in $r$.
\end{proof}

Theorem \ref{qhmt} is essentially due to Lovász, cf.\ \cite[Theorem 1*]{lovasz}. His formulation was couched in terms of the maximum rank of matrices in linear subspaces of $M_{m,n} = M_{m,n}(\mathbb{F})$. The connection to Theorem \ref{qhmt} comes from the fact that every pair $(v,w) \in \mathbb{F}^n\times\mathbb{F}^m$ induces a rank one linear transformation $wv^*: u \mapsto \langle u,v\rangle w$ from $\mathbb{F}^n$ to $\mathbb{F}^m$.

For any linear subspace $\mathcal{V}$ of $M_{m,n}$ and any linear subspace $E$ of $\mathbb{F}^n$, write
$$\mathcal{V}[E] = {\rm span}\{Av: A \in \mathcal{V}\,\,{\rm and}\,\, v \in E\} \subseteq \mathbb{F}^m.$$ Observe that if $\mathcal{V}_R$ is the span of $\{wv^*: (v,w) \in R\}$, then for any subset $S$ of $\mathbb{F}^n$ we have $${\rm span}(N(S)) = \mathcal{V}_R[{\rm span}(S)].$$

We can infer Lovász's version of Theorem \ref{qhmt} using the following fact.

\begin{prop}\label{lafact}
Let $A_1, \ldots, A_d$ be rank one matrices in $M_{m,n}$ with $A_i = w_iv_i^*$ ($1 \leq i \leq d$). Then $A_1 + \cdots + A_d$ has rank $d$ if and only if $v_1$, $\ldots$, $v_d$ are linearly independent in $\mathbb{F}^n$ and $w_1$, $\ldots$, $w_d$ are linearly independent in $\mathbb{F}^m$.
\end{prop}

\begin{proof}
If $v_1$, $\ldots$, $v_d$ are linearly dependent then their span has dimension at most $d - 1$ and the orthocomplement, ${\rm span}(v_1, \ldots, v_d)^\perp$, has dimension at least $n - (d - 1)$. This orthocomplement is contained in the kernel of each $A_i$ and hence is also contained in the kernel of $A_1 + \cdots + A_d$, showing that the rank of the sum is at most $d - 1$. Similarly, the image of $A_1 + \cdots + A_d$ is contained in the span of $w_1$, $\ldots$, $w_d$, so if those vectors are linearly dependent then the image of the sum can have dimension at most $d-1$, again showing that its rank is at most $d - 1$.

Conversely, suppose $v_1$, $\ldots$, $v_d$ and $w_1$, $\ldots$, $w_d$ are both linearly independent. Let $u$ be any vector in the kernel of $A_1 + \cdots + A_d$. Then $$\langle u, v_1\rangle w_1 + \cdots + \langle u, v_d\rangle w_d = 0,$$ so linear independence of the $w_i$'s entails that $$\langle u, v_1\rangle = \cdots = \langle u, v_d\rangle = 0,$$ and hence $u \in {\rm span}(v_1, \ldots, v_d)^\perp$. We conclude that $${\rm ker}(A_1 + \cdots + A_d) = {\rm span}(v_1, \ldots, v_d)^\perp$$ (as the reverse containment is trivial). Since the $v_i$'s are linearly independent the dimension of their span is $d$, so this shows that the kernel of $A_1 + \cdots + A_d$ has dimension exactly $n -d$, which implies that $A_1 + \cdots + A_d$ has rank $d$.
\end{proof}

\begin{cor} \label{qhmtl} (Lovász's theorem)
Let $m \geq n \geq 1$ and let $\mathcal{V}$ be a linear subspace of $M_{m,n}$. Suppose $\mathcal{V}$ is generated by rank one matrices. Then $\mathcal{V}$ contains a rank $n$ matrix if and only if $${\rm dim}(\mathcal{V}[E]) \geq {\rm dim}(E)$$ for every linear subspace $E$ of $\mathbb{F}^n$.
\end{cor}

\begin{proof}
If $\mathcal{V}$ contains a rank $n$ matrix $A$, then ${\rm ker}(A) = \{0\}$ and so $${\rm dim}(\mathcal{V}[E]) \geq {\rm dim}(A(E)) = {\rm dim}(E)$$ for every linear subspace $E$ of $\mathbb{F}^n$; this establishes the forward implication. For the reverse implication, since $\mathcal{V}$ is generated (i.e., spanned) by rank one matrices, it has the form $$\mathcal{V}_R = {\rm span}\{wv^*: (v,w) \in R\}$$ for some finite $R \subset \mathbb{F}^n \times \mathbb{F}^m$, and the comment made before the proposition just above then shows that the dimension inequality assumed here implies the one in Theorem \ref{qhmt}. Thus we can infer from that result the existence of a saturated matching $(v_1,w_1), \ldots, (v_n,w_n)$ in $R$. According to Proposition \ref{lafact} the matrix $w_1v_1^* + \cdots + w_nv_n^*$ is then the desired rank $n$ matrix in $\mathcal{V}_R$.
\end{proof}

The reverse implication from Corollary \ref{qhmtl} to  Theorem \ref{qhmt} can be drawn using another simple linear algebraic fact.

\begin{prop}\label{lafact2}
Let $A_1, \ldots, A_r$ be rank one matrices in $M_{m,n}$. Suppose that some linear combination of the $A_i$'s has rank $d$. Then there are $d$ distinct indices $i_1, \ldots, i_d$ such that $A_{i_1} + \cdots + A_{i_d}$ has rank $d$.
\end{prop}

\begin{proof}
Let $A$ be a linear combination of the $A_i$'s whose rank is $d$. Letting $P$ be the orthogonal projection of $\mathbb{F}^n$ onto ${\rm ker}(A)^\perp$ and $Q$ the orthogonal projection of $\mathbb{F}^m$ onto ${\rm im}(A)$, and replacing each $A_i$ with $QA_iP$ (or discarding it if $QA_iP = 0$), we can reduce to the case where $m = n = d$.

Now consider the polynomial $$p(x_1, \ldots, x_r) = {\rm det}(x_1A_1 + \cdots + x_rA_r).$$ This is a homogeneous polynomial of degree $d$ ($= m = n$). By hypothesis, it is not identically zero. Therefore there exist indices $i_1, \ldots, i_d$ such that the coefficient of $x_{i_1}\cdots x_{i_d}$ in $p(x_1, \ldots, x_d)$ is nonzero. Moreover, the determinant of $x_1A_1 + \cdots + x_rA_r$ is a multilinear function of the columns of $x_1A_1 + \cdots + x_rA_r$, and since any two columns of a rank one matrix are linearly dependent, the coefficient of any monomial with repeated indices will be zero. In particular, the indices $i_1$, $\ldots$, $i_d$ must be distinct. Setting $x_{i_1} = \cdots = x_{i_d} = 1$ and letting all other $x_i$'s be $0$, we see that ${\rm det}(A_{i_1} + \cdots + A_{i_d}) \neq 0$. So this sum has rank $d$.
\end{proof}

Using this result together with Proposition \ref{lafact}, we can infer Theorem \ref{qhmt} from Corollary \ref{qhmtl}.

Interpreting Theorem \ref{qhmt} in terms of matrices yields a slight improvement of that result. The vector space $M_{m,n}$ has dimension $mn$ over $\mathbb{F}$, so if $R$ contains more than $mn$ elements then the matrices $wv^*$ with $(v,w) \in R$ must be linearly dependent. And the equality ${\rm span}(N(S)) = \mathcal{V}_R[{\rm span}(S)]$ when $\mathcal{V}_R$ is the span of $\{wv^*: (v,w) \in R\}$, noted earlier, shows that only the span of the matrices $wv^*$ matters. Thus, even if $R$ is infinite, it can always be reduced to at most $mn$ pairs (whose corresponding rank one matrices are linearly independent) without affecting ${\rm span}(N(S))$ for any $S \subseteq \mathbb{F}^n$. Theorem \ref{qhmt} can then be applied to the reduced $R$. This yields the following.

\begin{cor}\label{imp}
In Theorem \ref{qhmt}, $R \subseteq \mathbb{F}^n\times \mathbb{F}^m$ need not be assumed finite.
\end{cor}

The original combinatorial formulation of Hall's theorem in \cite{hall} was framed in terms of ``transversals'' rather than perfect matchings: if $S_1, \ldots, S_n$ are sets, and for each $1 \leq k \leq n$ the union of any $k$ of them has cardinality at least $k$, then there exist distinct elements $s_1, \ldots, s_n$ such that $s_i \in S_i$ for all $i$. The marriage interpretation can be recovered by letting $S_i$ be the set of girls that the $i$th boy knows. Thus the graph-theoretic and combinatorial formulations are trivially equivalent. However, the combinatorial formulation suggests a different linear generalization, this one due to Rado \cite[Theorem 1]{rado}.

\begin{cor}\label{qhmtr} (Rado's theorem)
Let $m \geq n \geq 1$ and let $S_1$, $\ldots$, $S_n$ be subsets of $\mathbb{F}^m$. Suppose that for each $1 \leq k \leq n$, the linear span of the union of any $k$ of the $S_i$'s has dimension at least $k$. Then there are linearly independent vectors $w_1, \ldots, w_n \in \mathbb{F}^m$ with $w_i \in S_i$ ($1 \leq i \leq n$).
\end{cor}

\begin{proof}
Let $R \subset \mathbb{F}^n\times \mathbb{F}^m$ consist of all pairs $(e_i, v)$ such that $v \in S_i$. The dimension inequality of Theorem \ref{qhmt} then holds by essentially the same argument as the one used in the proof of Corollary \ref{qhmtc}. (If ${\rm dim}(S) = k$ then at least $k$ standard basis vectors are not orthogonal to $S$, so $N(S)$ contains the union of at least $k$ of the $S_i$'s.) Theorem \ref{qhmt}/Corollary \ref{imp} therefore yields the existence of a saturated matching $(v_1, w_1)$, $\ldots$, $(v_n,w_n)$, in which, after reordering, we can assume $v_i = e_i$ for all $i$. Then each $w_i$ belongs to $S_i$, and the vectors $w_1$, $\ldots$, $w_n$ are linearly independent.
\end{proof}

Rado's theorem is the result most commonly referenced in the literature as a ``linear marriage theorem''. It is intermediate in strength between Hall's original theorem and Theorem \ref{qhmt}. As we saw in the proof of Corollary \ref{qhmtc}, Hall's theorem is equivalent to the special case of Theorem \ref{qhmt} when $R$ is contained in $\{e_1, \ldots, e_n\} \times \{e_1, \ldots, e_m\}$; and we see from the proof of Corollary \ref{qhmtr} that Rado's theorem is equivalent to the special case of Theorem \ref{qhmt} when $R$ is contained in $\{e_1, \ldots, e_n\} \times \mathbb{F}^m$. (But I should add that one can also infer Theorem \ref{qhmt} from the matroid generalization Rado gives in \cite[Theorem 3]{rado}.)

\subsection{Kőnig's theorem}

We will need the following strengthening of Theorem \ref{qhmt}/Corollary \ref{imp} that allows for a ``deficiency'' in the condition ${\rm dim}\, {\rm span}(N(S)) \geq {\rm dim}\, {\rm span}(S)$.

\begin{cor}\label{dqhmt}
Let $m,n \geq 1$ and let $R \subseteq \mathbb{F}^n \times \mathbb{F}^m$. Suppose there exists $d \geq 0$ such that $${\rm dim}\, {\rm span}(N(S)) \geq {\rm dim}\, {\rm span}(S) - d$$ for every finite subset $S$ of $\mathbb{F}^n$. Then there exist $(v_1,w_1), \ldots, (v_{n-d},w_{n-d}) \in R$ such that $v_1$, $\ldots$, $v_{n-d}$ are linearly independent in $\mathbb{F}^n$ and $w_1$, $\ldots$, $w_{n-d}$ are linearly independent in $\mathbb{F}^m$.
\end{cor}

\begin{proof} 
Replace $\mathbb{F}^m$ with $\mathbb{F}^{m+d}$ and add to $R$ all pairs $(e_i, f_j)$ with $1 \leq i \leq n$ and $m+1 \leq j \leq m + d$, where $(e_i)$ and $(f_j)$ are the standard bases of $\mathbb{F}^n$ and $\mathbb{F}^{m+d}$. This restores the inequality ${\rm dim}\, {\rm span}(N(S)) \geq {\rm dim}\, {\rm span}(S)$, so that Theorem \ref{qhmt}/Corollary \ref{imp} can be applied, yielding a saturated matching in the sense of that result. That saturated matching will include at most $d$ pairs of the form $(e_i,f_j)$ with $m+1 \leq j \leq m + d$, which can then be discarded.
\end{proof}

The classical theorem of Kőnig \cite{konig} relates partial matchings of bipartite graphs to vertex covers. The relevant definitions for us go as follows. Let $R \subseteq \mathbb{F}^n \times \mathbb{F}^m$. A {\it cover} for $R$ is a pair $(E,F)$ of linear subspaces of $\mathbb{F}^n$ and $\mathbb{F}^m$ for which every pair $(v,w)$ in $R$ satisfies either $v \in E$ or $w \in F$ (or both). Its {\it size} is the value ${\rm dim}(E) + {\rm dim}(F)$. A {\it matching} in $R$ is a finite subset $\{(v_1, w_1), \ldots, (v_s, w_s)\}$ of $R$ such that the vectors $v_1$, $\ldots$, $v_s$ are linearly independent in $\mathbb{F}^n$ and the vectors $w_1$, $\ldots$, $w_s$ are linearly independent in $\mathbb{F}^m$. Its {\it size} is the number $s$.

\begin{thm}\label{qkt} (Linear Kőnig's theorem)
Let $m,n \geq 1$ and let $R \subseteq \mathbb{F}^n \times \mathbb{F}^m$. Then the maximum size of a matching in $R$ equals the minimum size of a cover for $R$.
\end{thm}

\begin{proof}
It is easy to see that the size of any matching is less than or equal to the size of any cover. For supppose $\{(v_1, w_1), \ldots, (v_s, w_s)\}$ is a matching and $(E,F)$ is a cover. Let $I_1 = \{1 \leq i \leq s: v_i \in E\}$ and $I_2 = \{1 \leq i \leq s: v_i \not\in E\}$. Then we must have $w_i \in F$ for every $i \in I_2$, so $s = |I_1| + |I_2| \leq {\rm dim}(E) + {\rm dim}(F)$.

Conversely, let $(E,F)$ be a cover of minimum size and let $d = n - ({\rm dim}(E) + {\rm dim}(F))$, so that $n - d$ is the minimum size of a cover. I claim that ${\rm dim}\, {\rm span}(N(S)) \geq {\rm dim}\, {\rm span}(S) - d$ for every finite subset $S \subset \mathbb{F}^n$. If this were to fail for some $S$, then we could let $E' = {\rm span}(S)^\perp$ and $F' = {\rm span}(N(S))$. Then ${\rm dim}(E') = n - {\rm dim}\, {\rm span}(S)$ and, by assumption, ${\rm dim}(F') < {\rm dim}\, {\rm span}(S) - d$, so that
${\rm dim}(E') + {\rm dim}(F') < n-d$. But $(E',F')$ is easily seen to be a cover for $R$, contradicting minimality of $(E,F)$.

We can therefore apply Corollary \ref{dqhmt} and obtain a matching of size $n - d$, which is the minimum size of a cover.
\end{proof}

Conversely, Theorem \ref{qhmt} can be inferred from Theorem \ref{qkt}. Assume Theorem \ref{qkt} is known and let $R \subseteq \mathbb{F}^n \times \mathbb{F}^m$. Suppose there is no saturated matching. Then the maximum size of a matching in $R$ is at most $n - 1$, so according to Theorem \ref{qkt} there exists a cover $(E,F)$ with ${\rm dim}(E) + {\rm dim}(F) \leq n - 1$. Letting $S$ be a basis for $E^\perp \subseteq \mathbb{F}^n$, we then have ${\rm dim}\, {\rm span}(S) = n - {\rm dim}(E)$ and ${\rm dim}\, {\rm span}(N(S)) \leq {\rm dim}(F)$ since $N(S) \subseteq F$ by virtue of the fact that $(E,F)$ is a cover. But by the above ${\rm dim}(F) \leq n - 1 - {\rm dim}(E) = {\rm dim}\, {\rm span}(S) - 1$, so we have proven the contrapositive of ``${\rm dim}\, {\rm span}(N(S)) \geq {\rm dim}\, {\rm span}(S)$ for every finite $S$ implies there is a saturated matching''.

Thus, the linear versions of the theorems of Hall and Kőnig are related in the same way as their classical counterparts.

The classical theorem of Kőnig easily follows from Theorem \ref{qkt}. A {\it cover} for a relation $r \subseteq \{1, \ldots, n\} \times \{1, \ldots, m\}$ is a pair $(S, T)$ of subsets of $\{1, \ldots, n\}$ and $\{1, \ldots, m\}$ for which every pair $(i,j) \in r$ satisfies either $i \in S$ or $j \in T$ (or both). Its {\it size} is the value $|S| + |T|$. A {\it matching} in $r$ is a set of pairs $(i_1, j_1), \ldots, (i_s, j_s) \in r$ such that the $i$'s are distinct and the $j$'s are distinct. Its {\it size} is the number $s$.

\begin{cor} (Kőnig's theorem)
Let $m, n \geq 1$ and let $r \subseteq \{1, \ldots, n\}\times \{1, \ldots, m\}$. Then the maximum size of a matching in $r$ equals the minimum size of a cover for $r$.
\end{cor}

\begin{proof}
As in the proof of Corollary \ref{qhmtc}, let $R \subset \mathbb{F}^n \times \mathbb{F}^m$ be the set of pairs $(e_i, e_j)$ with $(i,j) \in r$. It is clear that matchings for $r$ correspond to matchings for $R$ in the sense of Theorem \ref{qkt}, and with equal sizes. It is also clear that any cover $(S,T)$ for $r$ generates a cover $(E,F)$ for $R$ with $E = {\rm span}\{e_i: i \in S\}$ and $F = {\rm span}\{e_i: i \in T\}$, again with equal sizes. There may be other covers for $R$, but a moment's thought shows that every cover for $R$ contains one of this form: since all the vectors appearing in $R$ belong to the standard basis $(e_i)$, we can reduce any cover $(E, F)$ to the spans of the standard basis vectors $E$ and $F$ contain, without affecting the covering property. We can now infer the desired conclusion from Theorem \ref{qkt}.
\end{proof}

In the last section I attributed Corollary \ref{qhmtl} to Lovász. Theorem 1* of \cite{lovasz} is actually a little stronger than that result in that it accommodates a possible defect in the dimension of $\mathcal{V}[E]$ versus the dimension of $E$. This can also be directly deduced from Corollary \ref{dqhmt}.

\begin{cor}\label{locor} \cite[Theorem 1*]{lovasz}
Let $m, n \geq 1$ and let $\mathcal{V}$ be a linear subspace of $M_{m,n}$. Suppose $\mathcal{V}$ is generated by rank one matrices and let $$d = {\rm max}\{{\rm dim}(E) - {\rm dim}(\mathcal{V}[E]): E\mbox{ is a linear subspace of }\mathbb{F}^n\}.$$ Then the maximum rank of a matrix in $\mathcal{V}$ is $n - d$.
\end{cor}

\begin{proof}
Let $r$ be the maximum rank of a matrix in $\mathcal{V}$ and let $A \in \mathcal{V}$ have this rank. Then ${\rm dim}\, {\rm ker}(A) = n - r$. Thus any linear subspace $E$ of $\mathbb{F}^n$ satisfies
\begin{eqnarray*}
{\rm dim}(\mathcal{V}[E]) &\geq& {\rm dim}(A(E))\\
&\geq& {\rm dim}(E) - {\rm dim}\, {\rm ker}(A)\\
&=& {\rm dim}(E) - (n - r),
\end{eqnarray*}
so that $n - r \geq d$, or equivalently, $r \leq n - d$. Conversely, say $\mathcal{V} = \mathcal{V}_R$ for some $R \subseteq \mathbb{F}^n\times\mathbb{F}^m$. For any finite subset $S$ of $\mathbb{F}^n$, if $E = {\rm span}(S)$ then $${\rm dim}\, {\rm span}(N(S)) = {\rm dim}(\mathcal{V}[E]) \geq {\rm dim}(E) - d = {\rm dim}\, {\rm span}(S) - d,$$ so according to Corollary \ref{dqhmt} there is a matching in $R$ of size $n - d$; and it then follows from Proposition \ref{lafact} that $\mathcal{V}$ contains a matrix of rank $n - d$. So we also have the reverse inequality $r \geq n - d$.
\end{proof}

Alternatively, one can deduce the defect form of Lovász's result (Corollary \ref{locor}) from its full rank version (Corollary \ref{qhmtl}) in the same way that we deduced the defect form of the linear marriage theorem (Corollary \ref{dqhmt}) from the linear marriage theorem (Theorem \ref{qhmt}), i.e., by adding $d$ extra ``dummy'' dimensions in the range in order to restore the full rank condition.

\subsection{Dilworth's theorem}

The theorems of Hall and Kőnig belong to a circle of ideas in combinatorics that also includes results such as theorems of Dilworth and Menger, to which they may be said to be ``equivalent'' in the sense that each of these results is more or less easily deduced from the others. Depending on how loosely one interprets ``easily deduced'', this circle of ideas could be expanded further \cite{cameron}. 

Dilworth's theorem states that the minimum number of chains needed in a partition of a finite poset into chains equals the maximum cardinality of an antichain in the poset. This is a little more distant from the theorems of Hall and Kőnig then they are to each other, because it deals with partially ordered sets rather than bipartite graphs. So before we can formulate linear and matrix analogs of it, we first need to identify the appropriate linear and matrix versions of a partially ordered set. The latter was done in \cite{weaver}: in finite dimensions, the matrix version of a strict partial order (``strict'' meaning that we axiomatize $\succ$, not $\succeq$) is a nilpotent operator algebra, i.e., a linear subspace $\mathcal{V}$ of $M_n$ ($= M_{n,n}(\mathbb{F})$) which satisfies $\mathcal{V}^2 \subseteq \mathcal{V}$ and $\mathcal{V}^n = \{0\}$.

The linear version of this definition goes as follows. Let a {\it linorder} be a relation $R \subset \mathbb{F}^n \times \mathbb{F}^n$ with the properties (1) if $(v,w)$ and $(\tilde{v}, \tilde{w})$ both belong to $R$, and $w \not\perp \tilde{v}$, then $(v, \tilde{w})$ belongs to $R$; (2) $(v,w) \in R$ implies $v \perp w$. The first condition is a version of transitivity, and the second is a version of antisymmetry in the form $a \succ b \to a \neq b$. It is routine to check that ${\rm span}\{wv^*: (v,w) \in R\}$ is a nilpotent operator algebra, the key point being that $\tilde{w}\tilde{v}^*\cdot wv^*$ is zero if $\langle w, \tilde{v}\rangle = \tilde{v}^*w = 0$ and a nonzero multiple of $\tilde{w}v^*$ if $\langle w, \tilde{v}\rangle  = \tilde{v}^*w \neq 0$. Conversely, every rank one generated nilpotent operator algebra is the linear span of a finite relation $R$ with the above properties.

A {\it bi-chain} for a linorder is a sequence $(w_1, v_1, \ldots, w_r, v_r)$ such that $w_i \not\perp v_i$ for $1 \leq i \leq r$ and $(v_i, w_{i+1}) \in R$ for $1 \leq i < r$. Its {\it length} is the value $r$. A {\it decomposition of $\mathbb{F}^n$ into bi-chains} is a family of bi-chains whose $v$'s collectively form a basis and whose $w$'s collectively form a basis. Its {\it size} is the number of bi-chains in the family. An {\it antichain} for $R$ is a linear subspace $C$ of $\mathbb{F}^n$ for which every $(v,w) \in R$ satisfies $v \perp C$ or $w \perp C$ (or both). Equivalently, $P\mathcal{V}_R P = \{0\}$ where $P$ is the orthogonal projection onto $C$.

Recall from the last section that a cover for $R$ is a pair of linear subspaces $E$ and $F$ of $\mathbb{F}^n$ such that every $(v,w) \in R$ satisfies either $v \in E$ or $w \in F$ (or both), and its size is the value ${\rm dim}(E) + {\rm dim}(F)$.

\begin{lem}\label{dillem}
Let $n \geq 1$ and let $R \subset \mathbb{F}^n\times\mathbb{F}^n$ be a linorder. Then the maximum dimension of an antichain for $R$ equals $n$ minus the minimum size of a cover for $R$.
\end{lem}

\begin{proof}
Let $(E,F)$ be a cover for $R$ of minimum size. Then $C = E^\perp \cap F^\perp = (E + F)^\perp$ is an antichain for $R$, and the size of the cover is $${\rm dim}(E) + {\rm dim}(F) = {\rm dim}(E + F) + {\rm dim}(E \cap F) = n - {\rm dim}(C) + {\rm dim}(E \cap F).$$ So ${\rm dim}(C) \geq n - ({\rm dim}(E) + {\rm dim}(F))$. This shows that the maximum dimension of an antichain is at least $n$ minus the minimum size of a cover.

Conversely, let $C$ be an antichain for $R$ of maximum dimension; we must find a cover whose size is exactly $n - {\rm dim}(C)$. The desired cover is $F = {\rm span}(N(C))$ and $E = (C + F)^\perp$. First, we evaluate its size. Since $(v,w) \in R$ contributes to $N(C)$ only if $v \not\perp C$, the antichain property requires $w \perp C$ for every such pair. This shows that $N(C) \perp C$, and hence ${\rm dim}(C + F) = {\rm dim}(C) + {\rm dim}(F)$ and therefore $${\rm dim}(E) + {\rm dim}(F) = (n - {\rm dim}(C) - {\rm dim}(F)) + {\rm dim}(F) = n - {\rm dim}(C).$$ So $(E,F)$ has the correct size, and we must now show that it is a cover.

Let $(v,w) \in R$. Assume $v \not\in E$; we must prove $w \in F$. Since $v \not\in E$, either $v \not\perp C$ or $v \not\perp N(C)$. In the first case, we immediately get $w \in N(C) \subseteq F$. In the second case, there must exist $(\tilde{v},\tilde{w}) \in R$ such that $\tilde{v} \not\perp C$ and $\tilde{w} \not\perp v$. Then condition (1) in the definition of a linorder yields $(\tilde{v},w) \in R$, which again entails $w \in N(C) \subseteq F$. We conclude that $(E,F)$ is a cover, as desired.
\end{proof}

\begin{thm}\label{qdt} (Linear Dilworth's theorem)
Let $n \geq 1$ and let $R \subset \mathbb{F}^n \times \mathbb{F}^n$ be a linorder. Then the minimum size of a decomposition into bi-chains equals the maximum dimension of an antichain.
\end{thm}

\begin{proof}
Let $(E,F)$ be a cover of minimum size and suppose we are given a decomposition into bi-chains. Letting $a$ be the total number of $v$'s, over all the bi-chains in the decomposition, that belong to $E$, and letting $b$ be the total number of $w$'s that belong to $F$, we get that $a \leq {\rm dim}(E)$ and $b \leq {\rm dim}(F)$ (since the $v$'s are linearly independent, as are the $w$'s). And for each bi-chain $(w_1, v_1, \ldots, w_r, v_r)$ in the decomposition, we know that either $v_i \in E$ or $w_{i+1} \in F$, for each $1 \leq i < r$, which implies that the number of $v$'s in the bi-chain that belong to $E$ plus the number of $w$'s in the bi-chain that belong to $F$ is at least $r - 1$. The sum of the lengths (values of $r$) of the bi-chains in the decomposition equals $n$. Thus if $c$ is the total number of bi-chains in the decomposition, we get $${\rm dim}(E) + {\rm dim}(F) \geq a + b \geq n - c.$$ According to the lemma, it follows that $c$, the size of an arbitrary decomposition into bi-chains, is greater than or equal to the maximum dimension of an antichain.

Conversely, we can apply Theorem \ref{qkt} to $R$ and get that the maximum size of a matching equals the minimum size of a cover equals $n$ minus the maximum dimension of an antichain. Thus there is a matching $\{(v_1, w_1), \ldots, (v_s, w_s)\}$ with $s = n -$ (the maximum dimension of an antichain). Find vectors $v_{s+1}, \ldots, v_n$ and $w_{s+1}, \ldots, w_n$ such that $v_1, \ldots, v_n$ and $w_1, \ldots, w_n$ are bases of $\mathbb{F}^n$. Then let $\phi: \{1, \ldots, n\} \to \{1, \ldots, n\}$ be a bijection with the property that $w_i \not\perp v_{\phi(i)}$ for all $i$; this is possible by the classical marriage theorem, because any $k$ of the vectors $w_1, \ldots, w_n$ span a $k$-dimensional linear subspace and hence cannot all be orthogonal to more than $n-k$ of the linearly independent vectors $v_1, \ldots, v_n$ --- i.e., they must collectively be non-orthogonal to at least $k$ of the $v$'s.

We can now define a directed bipartite graph on the vertex set $\{v_1, \ldots, v_n$, $w_1, \ldots, w_n\}$ by including an edge from $w_i$ to $v_{\phi(i)}$ for each $1 \leq i \leq n$ and an edge from $v_i$ to $w_i$ for each $1 \leq i \leq s$. Observe that there is exactly one edge from each $w$ to some $v$, and at most one edge from each $v$ to some $w$. Moreover, it follows from the definition of a linorder that this graph has no cycles. For if there is an edge from $v_i$ to $w_i$, from $w_i$ to $v_{\phi(i)}$, and from $v_{\phi(i)}$ to $w_{\phi(i)}$, then $(v_i, w_{\phi(i)})$ belongs to $R$ by condition (1) in the definition of a linorder, and inductively $(v_i, w_j)$ belongs to $R$ for any $w_j$ further down the path. So $v_i$ is orthogonal to all $w_j$'s which can be reached from it by a directed path, by condition (2) in the definition of a linorder. Then from the definition of $\phi$ we cannot have $i = \phi(j)$ for any such $j$.

The graph on $\{v_1, \ldots, v_n, w_1, \ldots, w_n\}$ therefore decomposes into a set of maximal paths, each beginning with a $w$ vector (since every $v$ vector is pointed to by a $w$ vector) and ending with a $v$ vector (since every $w$ vector points to a further $v$ vector). Each of these paths is evidently a bi-chain, so that we have described a decomposition into bi-chains. As the graph has $2n$ vertices and $n + s$ edges, the number of bi-chains in this decomposition is $2n - (n + s) = n - s$. Thus we have found a decomposition into bi-chains of size $n - s =$ the maximum dimension of an antichain.
\end{proof}

We can also go in the reverse direction and deduce Theorem \ref{qkt} (linear Kőnig) from Theorem \ref{qdt} (linear Dilworth). Assume Theorem \ref{qdt} is known and let $R \subseteq \mathbb{F}^n \times \mathbb{F}^m$. Then define a new relation $R' \subset \mathbb{F}^{n + m} \times \mathbb{F}^{n + m}$ by letting $R' = \{(v \oplus 0,0 \oplus w): (v,w) \in R\}$. This is easily seen to be a linorder. Every bi-chain has length either $1$ or $2$, and the bi-chains of length $2$ in any decomposition for $R'$ determine a matching for $R$; conversely, every matching of size $s$ for $R$ can be converted into a decomposition into bi-chains for $R'$ of size $(n + m) - s$, with $s$ bi-chains of length $2$ and $(n + m) - 2s$ bi-chains of length 1. So if $s$ is the maximum size of a matching, then the minimum size of a decomposition into bi-chains is $(n + m) - s$, and by Theorem \ref{qdt} this is the maximum dimension of an antichain. Then by Lemma \ref{dillem} the minimum size of a cover is $n + m - ((n + m) - s) = s =$ the maximum size of a matching.

The classical theorem of Dilworth can be inferred from Theorem \ref{qdt} as well.

\begin{cor} (Dilworth's theorem)
Let $P$ be a finite poset. Then the minimum size of a decomposition of $P$ into chains equals the maximum cardinality of an antichain in $P$.
\end{cor}

\begin{proof}
Let the vertices of $P$ be $p_1, \ldots, p_n$ and define a relation $R$ on $\mathbb{F}^n$ consisting of all pairs of the form $(e_i, e_j)$ such that $p_i \succ p_j$, where $(e_i)$ is the standard basis of $\mathbb{F}^n$. It is easy to see that this is a linorder. We show that the minimum size of a decomposition of $\mathbb{F}^n$ into bi-chains is equal to the minimum size of a partition of $P$ into chains. First, every chain $p_{i_1} \succ \cdots \succ p_{i_r}$ in $P$ can be converted into a bi-chain $(e_{i_1}, e_{i_1}, \ldots, e_{i_r}, e_{i_r})$ of the same length for $R$. In this way any partition of $P$ into chains yields a decomposition of $\mathbb{F}^n$ into bi-chains that has the same size. Conversely, suppose we are given a decomposition of $\mathbb{F}^n$ into $s$ many bi-chains of lengths $r_1$, $\ldots$, $r_s$. Every pair $(v_i, w_{i+1})$ ($1 \leq i < r_j$) in one of these bi-chains belongs to $R$, and hence both $v_i$ and $w_{i+1}$ are standard basis vectors, say $v_i = e_a$ and $w_{i+1} = e_b$. Create a directed graph on the vertex set of $P$ by placing an edge from $p_a$ to $p_b$ for every such pair. Note that $p_a \succ p_b$ in the originally given partial order. This graph has $\sum (r_j - 1) = n - s$ edges. Since the $v$'s, over all the bi-chains in the decomposition, are linearly independent and hence distinct, as are the $w$'s, every vertex in the graph has at most one edge coming in and at most one edge going out, and since the directed edges all point down the poset there are no cycles in the graph. This graph therefore decomposes $P$ into $n - (n - s) = s$ chains, showing that any decomposition of $\mathbb{F}^n$ into bi-chains can be converted into a partition of $P$ into chains that has the same size. We conclude that the minimum size of a decomposition of $\mathbb{F}^n$ into bi-chains equals the minimum size of a partition of $P$ into chains.

Next, we show that the maximum cardinality of an antichain in $P$ equals the maximum dimension of an antichain for $R$. First, if $S$ is an antichain in $P$ then $C = {\rm span}\{e_i: i \in S\}$ is an antichain for $R$ with ${\rm dim}(C) = |S|$. So for every antichain in $P$ there is an antichain for $R$ of the same size. Conversely, any antichain $C$ for $R$ can be enlarged to an antichain $C' = {\rm span}\{e_i: e_i \not\perp C\}$, and then $\{p_i: e_i \in C'\}$ is an antichain in $P$ whose cardinality is greater than or equal to ${\rm dim}(C)$. So for every antichain for $R$ there is an antichain in $P$ that is at least as large. We conclude that the maximum cardinality of an antichain in $P$ equals the maximum dimension of an antichain for $R$.

The result now follows from Theorem \ref{qdt}.
\end{proof}

The intuition behind bi-chains is that we are following the path a vector takes as successive rank one matrices $wv^*$, with $(v,w) \in R$, are applied to it. Thus one starts with $w_1$, then applies $w_2v_1^*$ to it to get a nonzero multiple of $w_2$, then applies $w_3v_2^*$ to that result to get a nonzero multiple of $w_3$, etc. The multiples are always nonzero because $v_i^*w_i = \langle w_i, v_i\rangle \neq 0$ for all $i$.

However, the theorem would fail if we only considered ``$w$-chains'' $(w_1, \ldots, w_r)$. For example, consider the linorder $R = \{(e_1, e_2), (e_1, e_3), (e_1, e_4)\} \subset \mathbb{F}^4$ which corresponds to a four-element poset in which a single element lies above a three-element antichain. It has a three-dimensional antichain, namely ${\rm span}\{e_2, e_3, e_4\}$, and in accordance with Theorem \ref{qdt}, there is a decomposition into three bi-chains. But there are also two $w$-chains $(e_1, e_2)$ and $(e_1 + e_3, e_4)$ whose elements constitute a basis of $\mathbb{F}^4$. In other words, if we only require that the $w$ vectors form a basis, and not also the corresponding $v$ vectors, then there can be chain decompositions of size strictly smaller than the maximum dimension of an antichain. By reversing the order in this example we can create a similar situation where the elements of two $v$-chains constitute a basis but their corresponding $w$'s do not.

Nonetheless, the more intuitive idea of decomposing into chains, rather than bi-chains, can be rescued in the following way. Given any linear subspace $\mathcal{V} \subseteq M_n$, let a {\it coherent chain in $\mathbb{F}^n$ relative to $\mathcal{V}$} be a sequence of the form $(v, Av, \ldots, A^{k-1} v)$ for some $v \in \mathbb{F}^n$ and some $A \in \mathcal{V}$. It is ``coherent'' because a single matrix in $\mathcal{V}$ implements every step. Let a {\it coherent decomposition of $\mathbb{F}^n$ into chains relative to $\mathcal{V}$} be a collection of coherent chains which are all implemented by one single $A \in \mathcal{V}$, and whose elements collectively form a basis of $\mathbb{F}^n$. Its {\it size} is the number of coherent chains in the family.

\begin{thm}\label{cqdt} (Coherent linear Dilworth's theorem)
Let $n \geq 1$ and let $R \subset \mathbb{F}^n\times \mathbb{F}^n$ be a linorder. Then the minimum size of a coherent decomposition into chains relative to $\mathcal{V}_R$ equals the maximum dimension of an antichain for $R$.
\end{thm}

\begin{proof}
As we noted earlier, $\mathcal{V}_R$ is a nilpotent algebra. We know from Propositions \ref{lafact} and \ref{lafact2} that the maximum rank $s$ of a matrix in $\mathcal{V}_R$ equals the maximum size of a matching as defined in Section 1.2, which equals the minimum size of a cover by Theorem \ref{qkt}, and therefore also equals $n$ minus the maximum dimension of an antichain by Lemma \ref{dillem}. 

Let $A \in \mathcal{V}_R$ be a matrix with maximum rank and put it in Jordan normal form. Since every matrix in $\mathcal{V}_R$ is nilpotent, $A$ has no nonzero eigenvalues and hence every block of its Jordan form has $1$'s along the first superdiagonal and $0$'s everywhere else. Each block contributes one dimension to ${\rm ker}(A)$, so there are $n - s$ blocks in all. For each block there is a nonzero vector $v$ such that the vectors $v$, $Av$, $\ldots$, $A^{k-1}v$ are linearly independent and $A^k v = 0$, where $k$ is the size of the block. Collecting these sequences over all the blocks of $A$ then yields a coherent decomposition of $\mathbb{F}^n$ into $n - s$ many chains, where, again, $n - s$ is also the maximum dimension of an antichain. This shows that there is a coherent decomposition of size equal to the maximum dimension of an antichain.

Conversely, given any coherent decomposition of $\mathbb{F}^n$ into $c$ many chains, let $B \in \mathcal{V}_R$ be the implementing matrix and let $\pi: \mathbb{F}^n \to \mathbb{F}^n/{\rm im}(B)$ be the quotient map. When the map $\pi$ is applied to any chain $(v, Bv, \ldots, B^{k-1}v)$ in the decomposition, it annihilates every term except possibly the first. But since all the elements of all the chains constitute a basis of $\mathbb{F}^n$, their projections must span the quotient space $\mathbb{F}^n/{\rm im}(B)$, which shows that this quotient has dimension at most $c$. Thus $n - {\rm rank}(B) \leq c$, and since ${\rm rank}(B) \leq s$ this shows that $c \geq n - s$, i.e., there are at least $n - s$ many chains in the decomposition. So the minimum size of a coherent decomposition into chains must exactly equal $n$ minus the maximum rank of a matrix in $\mathcal{V}_R$, i.e., the maximum dimension of an antichain.
\end{proof}

We can go from bi-chain decompositions to coherent decompositions in the following way. Given a decomposition of $\mathbb{F}^n$ into bi-chains $(w_1, v_1, \ldots, w_r, v_r)$, every ``interior'' pair $(v_i, w_{i+1})$ belongs to $R$ and hence the sum $A$, over all interior pairs of all bi-chains in the decomposition, of the matrices $wv^*$ will belong to $\mathcal{V}_R$; moreover, the linear independence of the $v$'s and $w$'s assumed in the definition of a decomposition into bi-chains entails via Proposition \ref{lafact} that $A$ must have rank equal to the number of summands, which is $n - c$ where $c$ is the number of chains in the decomposition. That means that putting $A$ in Jordan form yields a coherent decomposition of the same size as the original decomposition into bi-chains. So we can go from bi-chain decompositions to coherent decompositions, but I do not see any easy way to go in the reverse direction.

\subsection{Menger's theorem}

Menger's theorem belongs to the same circle of ideas as the theorems of Hall, Kőnig, and Dilworth. In one version it states that the maximum number of vertex-disjoint paths linking two sets of vertices in a directed graph equals the minimum number of vertices that must be removed to disconnect those two sets.

A linearized formulation of the notion of severing all paths between two sets of vertices in a directed graph goes as follows. Given linear subspaces $E$ and $F$ of $\mathbb{F}^n$ and a relation $R \subseteq \mathbb{F}^n \times \mathbb{F}^n$, define an {\it $(E,F)$-separator} (relative to $R$) to be a pair of linear subspaces $\tilde{E}$, $\tilde{F}$ with the properties that (1) $E \subseteq \tilde{E}$ and $F \subseteq \tilde{F}$, (2) $\tilde{F}^\perp \subseteq \tilde{E}$, and (3) for every $(v,w) \in R$, either $v \in \tilde{F}$ or $w \in \tilde{E}$. Condition (2) ensures that $\mathbb{F}^n$ decomposes into an orthogonal sum $\tilde{F}^\perp \oplus (\tilde{E} \cap \tilde{F}) \oplus \tilde{E}^\perp$, and condition (3) ensures that we cannot directly jump from the first summand to the last --- that is, $\mathcal{V}_R[\tilde{F}^\perp] \subseteq \tilde{E} = \tilde{F}^\perp \oplus (\tilde{E} \cap \tilde{F})$. Let the {\it size} of the separator $(\tilde{E}, \tilde{F})$ be the dimension of $\tilde{E} \cap \tilde{F}$.

Condition (3) can also be interpreted as saying that $\tilde{F}$ and $\tilde{E}$ constitute a cover for $R$ in the sense of Section 1.2.

Menger's theorem does not straightforwardly generalize to the linear setting because nonorthogonality can allow multiple independent paths to squeeze through a one-dimensional bottleneck. To illustrate this point, let a {\it bi-path from $E$ to $F$} be a sequence $(w_1, v_2, \ldots, w_r, v_r)$ with (1) $w_1 \in E$ and $v_r \in F$; (2) $w_i \not\perp v_i$ for $1 \leq i \leq r$; and (3) $(v_i, w_{i+1}) \in R$ for $1 \leq i < r$. Say that a collection of bi-paths is {\it independent} if all the $v$'s appearing in them are linearly independent, as are all the $w$'s. A natural conjecture would be that the minimum size of an $(E,F)$-separator equals the maximum number of independent bi-paths from $E$ to $F$. But the following example shows that this can fail.

\begin{exmp}
Working in $\mathbb{F}^7$, let $E = {\rm span}(e_1, e_2)$ and $F = {\rm span}(e_6, e_7)$, and define $R = \{(e_1, e_3 + e_4), (e_2, e_3 - e_4), (e_4 + e_5, e_6), (e_4 - e_5, e_7)\}$. There is an $(E,F)$-separator with size $1$, namely $\tilde{E} = {\rm span}(e_1, e_2, e_3, e_4)$ and $\tilde{F} = {\rm span}(e_4, e_5, e_6, e_7)$. But the two bi-paths $(e_1, e_1, e_3 + e_4, e_4 + e_5, e_6, e_6)$ and $(e_2, e_2, e_3 - e_4, e_4 - e_5, e_7, e_7)$ from $E$ to $F$ are independent.
\end{exmp}

In order to avoid this kind of problem, it seems that we need to impose a form of ``coherence'' on the linking paths. We make the following definition.

\begin{defn}
Let $R \subseteq \mathbb{F}^n \times \mathbb{F}^n$ and let $E$ and $F$ be linear subspaces of $\mathbb{F}^n$. Let $\iota: E \to \mathbb{F}^n$ be the natural inclusion and let $\pi: \mathbb{F}^n \to F$ be the orthogonal projection. Then we define the {\it coherent path capacity between $E$ and $F$ relative to $R$}, denoted $CPC_R(E,F)$, to be the maximum value of $${\rm rank}\left[\begin{matrix}I - A& \iota\\ \pi&0\end{matrix}\right] - n$$ over all $A \in \mathcal{V}_R$. Here $\left[\begin{matrix}I - A& \iota\\ \pi&0\end{matrix}\right]$ represents a linear map from $\mathbb{F}^n \oplus E$ to $\mathbb{F}^n \oplus F$.
\end{defn}

This definition requires some justification. Note first that the maximum rank of $\left[\begin{matrix}I - A& \iota\\ \pi&0\end{matrix}\right]$ is achieved by all $A$ in a dense, relatively open subset of $\mathcal{V}_R$. It is both the maximum rank and the generic rank over $A \in \mathcal{V}_R$.

Similarly, $I - A$ is invertible for generic $A \in \mathcal{V}_R$. When this is the case we have $${\rm rank}\left[\begin{matrix}I - A& \iota\\ \pi&0\end{matrix}\right] = n + {\rm rank}(\pi\cdot(I - A)^{-1}\cdot \iota)$$ by the Guttman rank additivity formula, where $-\pi\cdot(I - A)^{-1}\cdot \iota$ is the Schur complement of $I - A$, so that $CPC_R(E, F)$ could also be defined as the maximum rank of $\pi\cdot (I - A)^{-1}\cdot \iota$ over all $A \in \mathcal{V}_R$ such that $I - A$ is invertible, or equivalently over all $A \in \mathcal{V}_R$ with $\|A\| < 1$ to ensure invertibility.

If $A$ is the adjacency matrix of a directed graph, then expanding $(I - A)^{-1}$ as a formal power series $I + A + A^2 + \cdots$ embodies all walks in the graph, and $\pi\cdot(I - A)^{-1}\cdot \iota$ can be interpreted as representing ``all walks from $E$ to $F$''. In combinatorics this is called the ``transfer matrix method'' \cite{stanley}.

Another comment is that $CPC_R(E,F)$ is the largest natural number $k$ such that there exist $v_1, \ldots, v_k \in E$ and $A \in \mathcal{V}_R$ such that the coherent paths $(v_j, Av_j, A^2 v_j, \ldots)$ ``independently route to $F$'' in the sense that the vectors $$y_j = \sum_{r = 0}^\infty \pi(A^r v_j) \in F$$ ($1 \leq j \leq k$) are linearly independent. That is just because $\sum \pi(A^r v_j) = \pi(\sum A^r v_j) = \pi\cdot(I - A)^{-1}v_j$, so the existence of $v_1$, $\ldots$, $v_k$ is equivalent to $\pi\cdot (I - A)^{-1}\cdot \iota$ having rank at least $k$.

We can also check that in the classical setting $CPC_R(E,F)$ does exactly equal the maximum number of vertex-disjoint paths.

\begin{prop}\label{gprop}
Let $G$ be a directed graph on the vertex set $\{g_1, \ldots, g_n\}$ and let $H$ and $K$ be sets of vertices of $G$. Let $E = {\rm span}\{e_i: g_i \in H\} \subseteq \mathbb{F}^n$ and $F = {\rm span}\{e_i: g_i \in K\} \subseteq \mathbb{F}^n$, and let $R = \{(e_i, e_j):$ there is an edge from $g_i$ to $g_j\} \subset \mathbb{F}^n\times \mathbb{F}^n$. Then $CPC_R(E, F)$ equals the maximum number of vertex-disjoint paths from $H$ to $K$ in $G$.
\end{prop}

\begin{proof}
Suppose we are given $k$ vertex-disjoint paths from $H$ to $K$. By truncating, we can assume that the last element of each of these paths is the only element of that path that belongs to $K$. Let $I_1 \subseteq H$ be the $k$ vertices at which these paths begin and let $I_2 \subseteq K$ be the $k$ vertices at which they end. Let $A_0 \in \mathcal{V}_R$ be the sum of the rank one matrices $e_je_i^*$ over all $i$ and $j$ such that $(g_i, g_j)$ is an edge in one of the $k$ paths. We will verify that the $(n+k)\times(n+k)$ submatrix of $\left[\begin{matrix}I - A_0& \iota\\ \pi&0\end{matrix}\right]$ with columns $\{1, \ldots, n\} \sqcup I_1$ and rows $\{1, \ldots, n\} \sqcup I_2$ has nonzero determinant.

Indeed, we can show that there is exactly one way to choose a nonzero entry out of each column in such a way that the choices all belong to distinct rows. This entails that there is exactly one nonzero term in the Leibniz formula for the determinant, and hence the determinant of the specified submatrix is nonzero. To see this, let $(g_{i_1}, 
\ldots, g_{i_r})$ be one of the $k$ given paths from $H$ to $K$. The column corresponding to $i_1$ in the right $\left[\begin{matrix}\iota \\ 0\end{matrix}\right]$ part of the submatrix has exactly one nonzero entry, in the $i_1$ row. So that is the entry that must be chosen from this column. Then the $i_1$ column in the left $\left[\begin{matrix} I - A_0 \\ \pi\end{matrix}\right]$ part has two nonzero entries, in the $i_1$ and $i_2$ rows, but we have already used the $i_1$ row so we must now choose the $i_2$ row. This continues up to the final $i_r$ entry, where the $i_r$ column has two nonzero entries, namely the copy of $i_r$ in $\{1, \ldots, n\}$ and the copy of $i_r$ in $I_2$; and we have to choose the copy of $i_r$ in $I_2$ because the first $i_r$ row was already used in the previous step. This shows how choosing a nonzero entry from a distinct row in each of the $i_1, \ldots, i_r \in \{1, \ldots, n\}$ columns can only be done in one way. The same is true for the indices that appear in each of the $k$ given paths from $H$ to $K$, and in any remaining column (whose corresponding vertex in $G$ does not belong to any of the $k$ paths) there is exactly one nonzero entry, arising from the $I$ contribution to the upper left $n\times n$ corner. So there is exactly one nonzero term in the Leibniz formula for the determinant, as promised. We have shown that the rank of $\left[\begin{matrix}I - A_0& \iota\\ \pi&0\end{matrix}\right]$ is at least $n + k$, so that $CPC_R(E, F) \geq k$.

Conversely, let $CPC_R(E, F) = l$, i.e., suppose the generic rank of $\left[\begin{matrix}I - A& \iota\\ \pi&0\end{matrix}\right]$ is $n + l$. By Guttman rank additivity, this entails that the generic rank of $\pi\cdot (I - A)^{-1}\cdot \iota$ is $l$. This matrix has dimensions $|K| \times |H|$, and generic rank $l$ implies that there are $l$-element subsets $H_0 \subseteq H$ and $K_0 \subseteq K$ such that $\pi_0\cdot (I-A)^{-1}\cdot \iota_0$ is generically nonsingular, where $\iota_0: E_0 \to \mathbb{F}^n$ is the natural inclusion and $\pi_0: \mathbb{F}^n \to F_0$ is the orthogonal projection, with $E_0 = {\rm span}\{e_i: g_i \in H_0\}$ and $F_0 = {\rm span}\{e_i: g_i \in K_0\}$. Equivalently, $\left[\begin{matrix}I - A& \iota_0\\ \pi_0&0\end{matrix}\right]$ is generically nonsingular. Thus there must exist at least one generically nonzero term $t$ in the Leibniz formula for its determinant. Create a new directed graph $G'$ on the same vertex set $\{g_1, \ldots, g_n\}$ by placing an edge from $g_i$ to $g_j$ if the $(i,j)$ entry of the upper left $n\times n$ block appears in $t$. In this graph, every vertex has at most one incoming edge and at most one outgoing edge. Since the diagonal entries of $I - A$ are nonzero, there could be loops in this directed graph, and there could also be cycles. But every vertex $g_i$ in $H_0$ has no incoming edge, because, as in the first part of the proof, the $i$th row has to be used in the $i$th column of the right $\left[\begin{matrix}\iota \\ 0\end{matrix}\right]$ part; similarly, every vertex in $K_0$ has no outgoing edge. Otherwise, every vertex with an ingoing edge must also have an outgoing edge, and vice versa. Since the initial vertices all belong to $H_0$ and the terminal vertices to $K_0$, there must therefore be $l$ vertex-disjoint paths in $G$ from $H$ to $K$.
\end{proof}

Now let us prove a linearized version of Menger's theorem. We require a general linear algebraic fact.

\begin{lem}\label{ranklem}
Let $\mathcal{F}$ be an infinite field, let $A \in M_{m,n}(\mathcal{F})$, $v \in \mathcal{F}^n$, and $w \in \mathcal{F}^m$, and let $x$ be an algebraically independent variable. Then the rank of $A + xwv^T$ over $\mathcal{F}(x)$ equals $$\min\left({\rm rank}_{\mathcal{F}}\left[\begin{matrix} A&w\end{matrix}\right], {\rm rank}_{\mathcal{F}}\left[\begin{matrix} A\\ v^T\end{matrix}\right]\right).$$
\end{lem}

\begin{proof}
The determinant of any square submatrix of $A + xwv^T$ is a polynomial in $x$, and hence is nonzero if and only if it is nonzero when evaluated at all but finitely many values of $x \in \mathcal{F}$. This shows that the rank of $A + xwv^T$ over $\mathcal{F}(x)$ equals the maximum, over all evaluations of $x$, of its rank over $\mathcal{F}$. But for any value of $x \in \mathcal{F}$ the columns of $A + xwv^T$ are contained in the span of the columns of $\left[\begin{matrix} A&w\end{matrix}\right]$ (i.e., they belong to ${\rm im}\left[\begin{matrix} A&w\end{matrix}\right]$), so the rank of the evaluated matrix is less than or equal to ${\rm rank}\left[\begin{matrix} A&w\end{matrix}\right]$. Similarly, for any value of $x \in \mathcal{F}$ the rows of $A + xwv^T$ are contained in the span of the rows of $\left[\begin{matrix} A\\v^T\end{matrix}\right]$ (i.e., they belong to ${\rm im} \left[\begin{matrix} A^T & v\end{matrix}\right]$), so the rank of the evaluated matrix is also less than or equal to ${\rm rank}\left[\begin{matrix} A\\ v^T\end{matrix}\right]$. This shows that $${\rm rank}_{\mathcal{F}(x)}(A + xwv^T) \leq {\rm min}\left({\rm rank}_\mathcal{F}\left[\begin{matrix} A&w\end{matrix}\right], {\rm rank}_{\mathcal{F}}\left[\begin{matrix} A\\ v^T\end{matrix}\right]\right).$$

Conversely, let $k = {\rm rank}(A)$. If $w \in {\rm im}(A)$ then $${\rm rank}_{\mathcal{F}}\left[\begin{matrix} A&w\end{matrix}\right] = k \leq {\rm rank}_{\mathcal{F}(x)}(A + xwv^T),$$ so we get the reverse inequality in this case. The same would be true if $v^T \in {\rm im}(A^T)$. Otherwise, if $w \not\in {\rm im}(A)$ and $v^T \not\in {\rm im}(A^T)$, then both $\left[\begin{matrix} A&w\end{matrix}\right]$ and $\left[\begin{matrix} A\\ v^T\end{matrix}\right]$ have rank $k + 1$. Find $\tilde{v} \in \mathcal{F}^n$ and $\tilde{w} \in \mathcal{F}^m$ such that $A\tilde{v} = \tilde{w}^TA = 0$ but $v^T\tilde{v}$ and $\tilde{w}^Tw$ are both nonzero. Relative to a basis of $\mathcal{F}^n$ whose first element is $\tilde{v}$ and a dual basis of $\mathcal{F}^m$ whose first element is $\tilde{w}^T$, the matrix of $A$ has all $0$'s in its leftmost column and topmost row, but the top left entry $a = \tilde{w}^Twv^T\tilde{v}$ of $wv^T$ is nonzero. Also, there is a $k\times k$ submatrix of $A$ that excludes the leftmost column and topmost row and whose determinant $b$ is nonzero. Evaluating the coefficient of $x$ in the $(k+1)\times(k+1)$ minor of $A + xwv^T$ that also includes the leftmost column and topmost row then yields $ab \neq 0$, so the rank of $A + xwv^T$ is also $k + 1$. Thus the equality holds in this case as well.
\end{proof}

\begin{prop}\label{rankprop}
Let $\mathcal{F}$ be an infinite field, let $A \in M_{m,n}(\mathcal{F})$, $v_i \in \mathcal{F}^n$, and $w_i \in \mathcal{F}^m$ ($1 \leq i \leq r$), and let $x_1$, $\ldots$, $x_r$ be algebraically independent variables. Then the rank of $A + \sum_{i=1}^r x_iw_iv_i^T$ over $\mathcal{F}(x_1, \ldots, x_r)$ equals $$\min_{S \subseteq \{1, \ldots, r\}} {\rm rank}_{\mathcal{F}} \left[\begin{matrix} A&W_S\\ V_{S^c}^T&0\end{matrix}\right]$$ where $W_S$ is the matrix with columns $w_i$, $i \in S$, and $V_{S^c}$ is the matrix with columns $v_i$, $i \not\in S$.
\end{prop}

\begin{proof}
Let $X$ be the diagonal matrix with entries $x_1, \ldots, x_r$ and for any $S \subseteq \{1, \ldots, r\}$ let $X_S$ and $X_{S^c}$ be the smaller diagonal matrices with only $i \in S$ or $i \not\in S$ entries. Then $A + \sum_{i=1}^r x_iw_iv_i^T = A + WXV^T$ where $V$ and $W$ are the matrices with columns $v_i$ and $w_i$ ($1 \leq i \leq r$). Factoring $$A + WXV^T = \left[\begin{matrix}I& W_{S^c}X_{S^c}\end{matrix}\right]\left[\begin{matrix}A& W_S\\ V_{S^c}^T&0\end{matrix}\right]\left[\begin{matrix}I\\ X_SV_S^T\end{matrix}\right]$$ then yields that the rank of $A + \sum_{i=1}^r x_iw_iv_i^T$ is less than or equal to the rank of $\left[\begin{matrix} A&W_S\\ V_{S^c}^T&0\end{matrix}\right]$. As $S$ was arbitrary, we get $${\rm rank}\left(A + \sum_{i=1}^r x_iw_iv_i^T\right) \leq \min_{S \subseteq \{1, \ldots, r\}} {\rm rank} \left[\begin{matrix} A&W_S\\ V_{S^c}^T&0\end{matrix}\right].$$

The proof of the reverse inequality goes by induction on $r$. The $r = 1$ case is Lemma \ref{ranklem}. Assuming the proposition holds for $r - 1$, we apply the lemma to $A' = A + \sum_{i=2}^r x_iw_iv_i^T$, $v_1$, and $w_1$ over the field $\mathcal{F}(x_2, \ldots, x_r)$. This yields that the rank of $A + \sum_{i=1}^r x_iw_iv_i^T$ over $\mathcal{F}(x_1, \ldots, x_r)$ equals $$\min\left({\rm rank}\left[\begin{matrix}A'& w_1\end{matrix}\right], {\rm rank} \left[\begin{matrix}A'\\ v_1^T\end{matrix}\right]\right),$$ with both ranks evaluated over $\mathcal{F}(x_2, \ldots, x_r)$. We can now apply the induction hypothesis to $$\left[\begin{matrix}A'& w_1\end{matrix}\right] = \left[\begin{matrix}A&w_1\end{matrix}\right] + \sum_{i=2}^r x_iw_i\left[\begin{matrix}v_i\\0\end{matrix}\right]^T$$ to get that the rank of this matrix equals $$\min_{S \subseteq \{2, \ldots, r\}} {\rm rank}\left[\begin{matrix}A&w_1&W_S\\ V_{S^c}^T&0&0\end{matrix}\right].$$ By similar reasoning, the rank of $\left[\begin{matrix}A'\\ v_1^T\end{matrix}\right]$ equals $$\min_{S \subseteq \{2, \ldots, r\}} {\rm rank}\left[\begin{matrix}A&W_S\\v_1^T&0\\V_{S^c}^T&0\end{matrix}\right].$$ Taking the minimum of these two formulas yields the minimum of $${\rm rank} \left[\begin{matrix} A&W_S\\ V_{S^c}^T&0\end{matrix}\right]$$ over all $S \subseteq \{1, \ldots, r\}$, with the first formula accounting for all $S$ containing $1$ and the second formula accounting for all $S$ not containing $1$.
\end{proof}

\begin{thm}\label{qmt} (Linear Menger's theorem)
Let $n \geq 1$, let $R \subseteq \mathbb{F}^n\times \mathbb{F}^n$, and let $E$ and $F$ be linear subspaces of $\mathbb{F}^n$. Then the coherent path capacity between $E$ and $F$ relative to $R$ equals the minimum size of an $(E,F)$-separator.
\end{thm}

\begin{proof}
Let $A \in \mathcal{V}_R$ and let $(\tilde{E}, \tilde{F})$ be an $(E,F)$-separator. Observe that $\left[\begin{matrix}I - A& \iota\\ \pi&0\end{matrix}\right]$ maps $\tilde{F}^\perp \oplus E$ into $\tilde{E} \oplus \{0\}$. Thus its kernel has dimension at least $${\rm dim}(\tilde{F}^\perp \oplus E) - {\rm dim}(\tilde{E}) = n - {\rm dim}(\tilde{F}) + {\rm dim}(E) - {\rm dim}(\tilde{E}).$$ Since its domain $\mathbb{F}^n \oplus E$ has dimension $n + {\rm dim}(E)$, its rank is therefore at most
\begin{eqnarray*}(n + {\rm dim}(E)) - (n - {\rm dim}(\tilde{F}) + {\rm dim}(E) - {\rm dim}(\tilde{E})) &=& {\rm dim}(\tilde{E}) + {\rm dim}(\tilde{F})\\ &=& n + {\rm dim}(\tilde{E} \cap \tilde{F}).\end{eqnarray*} Thus $${\rm rank}\left[\begin{matrix}I - A& \iota\\ \pi&0\end{matrix}\right] - n \leq {\rm dim}(\tilde{E} \cap \tilde{F}).$$ Since $A$, $\tilde{E}$, and $\tilde{F}$ were arbitrary, this shows that the coherent path capacity is less than or equal to the minimum size of an $(E,F)$-separator.

For the reverse inequality, first reduce $R$ to a finite set as we did in Corollary \ref{imp}. Say $R = \{(v_1, w_1), \ldots, (v_r, w_r)\}$. For each $1 \leq i \leq r$ let $\tilde{v}_i = v_i \oplus 0 \in \mathbb{F}^n \oplus E$ and $\tilde{w}_i = w_i \oplus 0 \in \mathbb{F}^n \oplus F$. Then for any $A \in \mathcal{V}_R$ we can write $$\left[\begin{matrix}I - A& \iota\\ \pi&0\end{matrix}\right] = \left[\begin{matrix}I& \iota\\ \pi&0\end{matrix}\right] + \sum_{i=1}^r x_i\tilde{w}_i\tilde{v}_i^*$$ for some $x_1, \ldots, x_r \in \mathbb{F}$. According to Proposition \ref{rankprop}, there is a subset $S \subseteq \{1, \ldots, r\}$ such that the generic rank of this matrix equals the rank of the matrix $$\left[\begin{matrix}I & \iota & W_S\\ \pi & 0 & 0\\ V_{S^c}^* & 0 & 0\end{matrix}\right]$$ where $W_S$ is the matrix with columns $w_i$, $i \in S$, and $V_{S^c}$ is the matrix with columns $v_i$, $i \not\in S$. By Guttman rank additivity, this rank equals $n$ plus the rank of $\left[\begin{matrix}\pi\\ V_{S^c}^*\end{matrix}\right]\left[\begin{matrix}\iota & W_S\end{matrix}\right]$. Thus $$CPC_R(E,F) = {\rm rank}\left(\left[\begin{matrix}\pi\\ V_{S^c}^*\end{matrix}\right]\left[\begin{matrix}\iota & W_S\end{matrix}\right]\right).$$

Let $C = E + {\rm span}\{w_i: i \in S\} = {\rm im}\left[\begin{matrix}\iota & W_S\end{matrix}\right]$ and $D = F + {\rm span}
\{v_i: i \not\in S\} = \left({\rm ker}\left[\begin{matrix}\pi\\ V_{S^c}^*\end{matrix}\right]\right)^\perp$. Then the rank of the product $\left[\begin{matrix}\pi\\ V_{S^c}^*\end{matrix}\right]\left[\begin{matrix}\iota & W_S\end{matrix}\right]$ equals ${\rm dim}(C) - {\rm dim}(C \cap D^\perp)$ (the dimension of the image of $\left[\begin{matrix}\iota & W_S\end{matrix}\right]$ minus the dimension that we lose when $\left[\begin{matrix}\pi\\ V_{S^c}^*\end{matrix}\right]$ is applied to it). We define $\tilde{E} = C + D^\perp$ and $\tilde{F} = D$. Then $\tilde{F}^\perp \subseteq \tilde{E}$ and so \begin{eqnarray*}{\rm dim}(\tilde{E} \cap \tilde{F}) &=& {\rm dim}(\tilde{E}) - {\rm dim}(\tilde{F}^\perp)\\ &=& {\rm dim}(C + D^\perp) - {\rm dim}(D^\perp)\\ &=&{\rm dim}(C) - {\rm dim}(C \cap D^\perp)\\ &=& CPC_R(E,F).\end{eqnarray*} As it is straightforward to verify that $\tilde{E}$ and $\tilde{F}$ constitute an $(E,F)$-separator, this shows that the coherent path capacity is greater than or equal to the minimum size of an $(E,F)$-separator. Thus we conclude that the two quantities are equal.
\end{proof}

The classical theorem of Menger is an easy consequence. If $H$ and $K$ be sets of vertices in a finite directed graph, an {\it $(H,K)$-separator} is a set of vertices $S$ such that every path that begins in $H$ and ends in $K$ must contain a vertex in $S$.

\begin{cor} (Menger's theorem)
Let $H$ and $K$ be sets of vertices in a finite directed graph $G$. Then the maximum number of vertex-disjoint paths from $H$ to $K$ equals the minimum cardinality of an $(H,K)$-separator.
\end{cor}

\begin{proof}
Place a loop at each vertex of $G$ and define $E$, $F$, and $R$ as in Proposition \ref{gprop}. We know from that result that the maximum number of vertex-disjoint paths from $H$ to $K$ equals $CPC_R(E,F)$, and we know from Theorem \ref{qmt} that $CPC_R(E,F)$ equals the minimum size of an $(E,F)$-separator. So we need only show that the minimum size of an $(E,F)$-separator (in the linear algebra sense) equals the minimum cardinality of an $(H,K)$-separator (in the graph theory sense).

Given an $(H,K)$-separator $S$, let $S'$ be the set of vertices reachable from $H$ in $G\setminus S$. Then ${\rm span}\{e_i: g_i \in S \cup S'\}$ and ${\rm span}\{e_i: g_i \in G\setminus S'\}$ constitute an $(E,F)$-separator whose size is $|S|$. Thus the minimum size of an $(E,F)$-separator is less than or equal to the minimum cardinality of an $(H,K)$-separator.

Conversely, suppose $\tilde{E}$ and $\tilde{F}$ constitute an $(E,F)$-separator. Define $\tilde{H} = \{g_i: e_i \not\in \tilde{F}\}$ and $\tilde{K} = \{g_i: e_i \not\in \tilde{E}\}$. Then $H \cap \tilde{K} = K \cap \tilde{H} = \emptyset$, and by condition (3) in the definition of an $(E,F)$-separator there are no edges from $\tilde{H}$ to $\tilde{K}$. So there is no path from $H$ to $K$ that lies entirely in $\tilde{H} \cup \tilde{K}$, i.e., $S = G\setminus (\tilde{H} \cup \tilde{K})$ is an $(H,K)$-separator. Also $|\tilde{H}| \geq {\rm dim}(\tilde{F}^\perp) = |G| - {\rm dim}(\tilde{F})$ and $|\tilde{K}| \geq {\rm dim}(\tilde{E}^\perp) = |G| - {\rm dim}(\tilde{E})$, so $$|S| = |G| - (|\tilde{H}| + |\tilde{K}|)  \leq {\rm dim}(\tilde{E}) + {\rm dim}(\tilde{F}) - |G| = {\rm dim}(\tilde{E} \cap \tilde{F}).$$ ($\tilde{H} \cap \tilde{K} = \emptyset$ since we began by placing a loop at each vertex. This forces every $e_i$ to belong to either $\tilde{E}$ or $\tilde{F}$.) Thus the minimum cardinality of an $(H,K)$-separator is less than or equal to the minimum size of an $(E,F)$-separator. We conclude that the two values are equal, and that completes the proof.
\end{proof}

Theorem \ref{qmt} also easily entails the linear Kőnig's theorem, Theorem \ref{qkt}. Given $R \subseteq \mathbb{F}^n \times \mathbb{F}^m$, define $R' \subset \mathbb{F}^{n+m} \times \mathbb{F}^{n+m}$ by setting $R' = \{(v \oplus 0, 0 \oplus w): (v,w) \in R\}$. Also let $E = \mathbb{F}^n \oplus \{0\}$ and $F = \{0\} \oplus \mathbb{F}^m$. It is straightforward to verify that any cover $(E_0, F_0)$ for $R$ generates an $(E,F)$-separator of the same size, namely $(\mathbb{F}^n \oplus F_0, E_0 \oplus \mathbb{F}^m)$, and every $(E,F)$-separator arises in this way from a cover for $R$. (This construction may help explain why $\tilde{E}$ and $\tilde{F}$ appear to have switched places in the definition of an $(E,F)$-separator, where the pair $(\tilde{F}, \tilde{E})$ is a cover, not $(\tilde{E}, \tilde{F})$.)
So the minimum size of an $(E,F)$-separator for $R'$ equals the minimum size of a cover for $R$. Also, every element of $\mathcal{V}_{R'}$ has the form $\left[\begin{matrix}0&0\\ A&0\end{matrix}\right]$ for some $A \in \mathcal{V}_R$, and vice versa, so the coherent path capacity between $E$ and $F$ equals the maximum value of $${\rm rank}\left[\begin{matrix}I_n&0&I_n\\ A&I_m&0\\ 0&I_m&0\end{matrix}\right] - (n+m)$$ over $A \in \mathcal{V}_R$. Subtracting the third column from the first and the third row from the second in this matrix yields $\left[\begin{matrix}0&0&I_n\\ A&0&0\\ 0&I_m&0\end{matrix}\right]$, whose rank is ${\rm rank}(A) + n + m$. But the maximum rank of $A \in \mathcal{V}_R$ is, by Propositions \ref{lafact} and \ref{lafact2}, precisely the maximum size of a matching for $R$. Thus the equality between coherent path capacity and the minimum size of an $(E,F)$-separator in Theorem \ref{qmt} entails the equality between the maximum size of a matching and the minimum size of a cover in Theorem \ref{qkt}.

\subsection{The Lindström–Gessel–Viennot lemma}

In the last section we noted that the coherent path capacity between $E \subseteq \mathbb{F}^n$ and $F \subseteq \mathbb{F}^n$, relative to $R = \{(v_1, w_1), \ldots, (v_r, w_r)\}$, is the maximum rank of $\pi\cdot (I - A)^{-1}\cdot \iota$ over all $A \in \mathcal{V}_R$ such that $I - A$ is invertible, with the intuition that $(I - A)^{-1} = I + A + A^2 + \cdots$ represents all walks in a graph. Equivalently, it is the rank over $\mathbb{F}(x_1, \ldots, x_r)$ of the matrix $\pi\cdot (I - WXV^*)^{-1}\cdot \iota$ where $V$ and $W$ are the $n \times r$ matrices with columns $v_i$ and $w_i$, and $X = {\rm diag}(x_1, \ldots, x_r)$ is a matrix of formal variables.

If $E$ and $F$ have the same dimension and the rank is full, we can give a formula for the determinant of $\pi\cdot (I - WXV^*)^{-1}\cdot \iota$ (and hence, by substitution, of $\pi\cdot (I - A)^{-1}\cdot \iota$ for any $A \in \mathcal{V}_R$). In some loose sense, this determinant counts the number of ways to choose vertex-disjoint paths from $E$ to $F$. Let $a_1, \ldots, a_k$ and $b_1, \ldots, b_k$ be vectors in $\mathbb{F}^n$ representing sources and sinks, and let $A, B \in M_{n,k}$ be the matrices whose columns are the $a_i$'s and the $b_i$'s. For any $S \subseteq \{1, \ldots, r\}$ let $V_S$ and $W_S$ be the $n \times |S|$ submatrices of $V$ and $W$ consisting of only those columns whose indices belong to $S$, let $X_S$ be the $|S| \times |S|$ matrix with diagonal entries $x_i$ for $i \in S$, and let $x_S = \prod_{i \in S} x_i$. Also let $G_S$ be the $(|S| + k) \times (|S| + k)$ matrix $$ G_S = \left[\begin{matrix} V_S^* W_S & V_S^* A \\ B^* W_S & B^* A \end{matrix}\right].$$

\begin{thm}\label{lgthm}
With the above notation, $$\det\Big( B^* (I_n - WXV^*)^{-1} A \Big) = \frac{ \displaystyle \sum (-1)^{|S|} x_S \det G_S }{ \displaystyle \sum (-1)^{|S|} x_S \det(V_S^* W_S)},$$ both sums being taken over all $S \subseteq \{1, \ldots, r\}$.
\end{thm}

\begin{proof}
Let $H(X)$ be the $(r + k) \times (r + k)$ matrix $$H(X) = \left[\begin{matrix} I_r - X V^* W & -X V^* A \\ B^* W & B^* A \end{matrix}\right].$$ We evaluate its determinant in two different ways. First, since $I_r - X V^* W$ evaluates to $I_r$ when $X = 0$, its determinant is nonzero and hence it is invertible over $\mathbb{F}(x_1, \dots, x_r)$. Applying the Schur complement factorization to this block yields \begin{eqnarray*}\det H(X) &=& \det(I_r - X V^* W) \det\Big( B^*A + B^*W (I_r - X V^* W)^{-1} X V^*A\Big)\\ &=& \det(I_r - X V^* W) \det\Big( B^* (I_n - W X V^*)^{-1} A \Big).\end{eqnarray*} Second, we expand $\det H(X)$ by multilinearity with respect to its top $r$ rows. For $1 \le i \le r$, the $i$-th row of $H(X)$ is $e_i^* - x_i u_i^*$, where $u_i^* = \left[\begin{matrix} v_i^* W & v_i^* A \end{matrix}\right]$. So we can decompose $\det H(X)$ into a sum, over all subsets $S$ of $\{1, \ldots, r\}$, of the determinant of the matrix formed from $H(X)$ by replacing the $i$th row, for $1 \leq i \leq r$, with $-x_i u_i^*$ when $i \in S$ and $e_i^*$ when $i \notin S$. Expanding the determinant via cofactors along the rows $i \not\in S$ effectively strikes out the $i$-th row and $i$-th column for all $i \notin S$. Because these elements lie on the main diagonal, each such expansion carries a positive sign. The remaining submatrix has dimensions $(|S| + k) \times (|S| + k)$ and corresponds to the row and column indices $S \cup \{r+1, \dots, r+k\}$. The top $|S|$ rows of this submatrix are the restrictions of $-x_i u_i^*$ for $i \in S$ to the columns $S \cup \{r+1, \dots, r+k\}$, which yields the matrix $-X_S \left[\begin{matrix} V_S^* W_S & V_S^* A \end{matrix}\right]$. The bottom $k$ rows are the restrictions of $\left[\begin{matrix} B^* W & B^* A \end{matrix}\right]$ to these same columns, yielding $\left[\begin{matrix} B^* W_S & B^* A \end{matrix}\right]$. Factoring the scalar $-x_i$ from each of the upper $|S|$ rows then extracts an overall coefficient of $\prod_{i \in S} (-x_i) = (-1)^{|S|} x_S$. The residual matrix is $G_S$. Thus summing over $S \subseteq \{1, \ldots, r\}$ yields $$ \det H(X) = \sum (-1)^{|S|} x_S \det G_S.$$

Finally, we evaluate the polynomial $\det(I_r - X V^* W)$. Applying the principal minor expansion yields $$\det(I_r - X V^* W) = \sum (-1)^{|S|} \det(M_S) = \sum x_S \det(V_S^* W_S),$$ where $M_S$ is the principal submatrix of $X V^* W$ indexed by $S$.

Equating the two expressions for $\det H(X)$ yields $$ \det(I_r - X V^* W) \det\Big( B^* (I_n - W X V^*)^{-1} A \Big) = \sum (-1)^{|S|} x_S \det G_S,$$ and the desired formula is then obtained by substituting the final expression for $\det(I_r - X V^* W)$.
\end{proof}

Let us say that $R \subseteq \mathbb{R}^n \times \mathbb{R}^n$ is {\it acyclic} if $\mathcal{V}_R^n = \{0\}$. Equivalently, every bi-path (see Section 1.4) has length at most $n$. In this case the formula in Theorem \ref{lgthm} simplifies.

\begin{cor}\label{lgcor} (Linear Lindström–Gessel–Viennot lemma)
If $R$ is acyclic then $$ \det \left( B^* \left( \sum_{j=0}^{n-1} (W X V^*)^j \right) A \right) \;\;=\;\; \sum (-1)^{|S|} x_S \det \left[\begin{matrix} V_S^* W_S & V_S^* A \\ B^* W_S & B^* A \end{matrix}\right],$$ taking the sum over all $S \subseteq \{1, \ldots, r\}$.
\end{cor}

\begin{proof}
By acyclicity, the infinite series $I -  WXV^* + (WXV^*)^2 - \cdots$ truncates at the $n$th power. So the left side of the equality in Theorem \ref{lgthm} becomes ${\rm det}(B^*(\sum_{j=0}^{n-1} (WXV^*)^j)A)$.

Create a directed graph on $\{1, \ldots, r\}$ by including an edge from $i$ to $j$ if $v_i \not\perp w_j$. Then acyclicity of $R$ guarantees that this graph is acyclic, so we can reorder the indices in such a way that $v_i \perp w_j$ whenever $i \geq j$. Having done this, the matrix $V^*W$ will be strictly upper triangular, and since $X$ is (even after reordering) diagonal, the matrix $XV^*W$ will also be strictly upper triangular. Therefore ${\rm det}(I_r - XV^*W) = 1$. This was the expression that evaluated to the denominator in Theorem \ref{lgthm}, so that denominator can be removed. This yields the stated result.
\end{proof}

When $R$ arises from an acyclic directed graph as in Proposition \ref{gprop}, Corollary \ref{lgcor} reduces to the classical Lindström–Gessel–Viennot lemma. To see this, let $G$ be a finite acyclic directed graph, let $k \geq 1$, and let $H$ and $K$ be sets of vertices in $G$ with $|H| = |K| = k$. Assign a weight $\omega_e$ to each edge $e$ in $G$, and for each directed path $P$ in $G$ let $\omega(P)$ be the product of the weights of the edges of the (by definition, non self-intersecting) path. For any vertices $g$ and $g'$, write $e(g, g')$ for the sum of $\omega(P)$ over all paths $P$ from $g$ to $g'$. The Lindström–Gessel–Viennot lemma evaluates the determinant of the matrix $M$ whose $(i,j)$ entry is $e(g_i, g_j')$ where $H = \{g_1, \ldots, g_k\}$ and $K = \{g_1', \ldots, g_k'\}$. It states that $${\rm det}(M) = \sum {\rm sign}(\sigma) \prod_{i=1}^k \omega(P_i)$$ where the sum is taken over all $k$-tuples of vertex disjoint paths $P_i$ from $H$ to $K$, and $\sigma$ is the permutation of $\{1, \ldots, k\}$ such that the path that begins at $g_i$ ends at $g_{\sigma(i)}'$.

As in Proposition \ref{gprop}, we consider $\mathbb{F}^n$ where $n = |G|$, with the standard basis vectors of $\mathbb{F}^n$ corresponding to the vertices of $G$. The vectors $a_i$ and $b_i$ are taken to be the elements of the standard basis that correspond to the vertices in $H$ and $K$, and each edge of $G$ contributes a pair of standard basis vectors $(v,w)$ to $R$, which correspond to the tail and head of that edge. The matrix $X = {\rm diag}(x_1, \ldots, x_r)$ represents the weights of the edges, and the matrices $A$ and $B$ are orthogonal projections onto the spans of the $a_i$'s and the $b_i$'s. Thus the left side of the equality in Corollary \ref{lgcor} is just the determinant of the matrix $M$, with $(WXV^*)^j$ contributing, in the entry corresponding to $g_i$ and $g_j'$, the sum of $\omega(P)$ over all paths of length $j$ linking $g_i$ to $g_j'$. The kind of reasoning used in the proof of Proposition \ref{gprop} shows that the determinant of the 0-1 matrix on the right side of the equality in Corollary \ref{lgcor} is zero unless $S$ consists of the edges in a $k$-tuple of vertex-disjoint paths from $H$ to $K$, in which case the determinant equals $(-1)^{|S|}\cdot {\rm sign}(\sigma)$, and $x_S$ recovers $\prod_{i=1}^k \omega(P_i)$. Thus Corollary \ref{lgcor} reduces to the Lindström–Gessel–Viennot lemma in this setting.

\section{The matrix setting}

\subsection{Noncommutative rank}

Now let $\mathcal{V}$ be any linear subspace of $M_{m,n} = M_{m,n}(\mathbb{F})$, with $\mathbb{F} = \mathbb{R}$ or $\mathbb{C}$, and recall the notation $\mathcal{V}[E] = {\rm span}\{Av: A \in \mathcal{V}, v \in E\}$ for $E$ a linear subspace of $\mathbb{F}^n$. According to Corollary \ref{qhmtl}, if $\mathcal{V}$ is generated by rank one matrices then the {\it no shrunk subspaces} condition, $${\rm dim}(\mathcal{V}[E]) \geq {\rm dim}(E)$$ for all linear subspaces $E$ of $\mathbb{F}^n$, ensures that $\mathcal{V}$ contains a rank $n$ matrix. However, as noted by Lovász \cite{lovasz}, this implication is not valid when one drops the assumption of rank one generation: for any $n > 1$ the space of $n\times n$ skew-symmetric matrices has no shrunk subspaces, but if $n$ is odd then every $n \times n$ skew-symmetric matrix has null determinant, i.e., is not invertible. Hence every such matrix has rank at most $n - 1$. Thus the naive generalization of the linear marriage theorem to arbitrary matrix spaces fails.

The relevant tool that is needed here is the notion of {\it noncommutative rank}, which we can define as $$\ncrank(\mathcal{V}) = \sup_{r \geq 1} \frac{1}{r} \cdot {\rm rank}(\mathcal{V} \otimes M_r)$$ where $\mathcal{V}\otimes M_r \subseteq M_{m,n}\otimes M_r \cong M_{mr,nr}$ and we define ${\rm rank}(\mathcal{W}) = \max\{{\rm rank}(A): A \in \mathcal{W}\}$ for any linear space of matrices $\mathcal{W}$. See the introduction to \cite{IQS} for several equivalent formulations of noncommutative rank.

For any $r \geq 1$, the quantity ${\rm rank}(\mathcal{V} \otimes M_r)$ is divisible by $r$ \cite[Lemma 5.6]{IQS0}. The quantity $\frac{1}{r}\cdot {\rm rank}(\mathcal{V} \otimes M_r)$ is nondecreasing for $r \geq \frac{n}{2} - 1$, but this can fail for smaller values of $r$ \cite[Propositions 1.12 and 1.13]{DM1}; it is constant and achieves its maximum for $r \geq n -1$ \cite[Theorem 1.8]{DM}. The noncommutative rank of $\mathcal{V}$ exactly equals $n - d$ where $d$ is the maximum of ${\rm dim}(E) - {\rm dim}(\mathcal{V}[E])$ over all linear subspaces $E$ of $\mathbb{F}^n$ \cite[Theorem 1]{FR}. Thus we obtain the following matrix version of Lovász's theorem (Corollary \ref{locor} above):

\begin{thm}\label{ncr}
Let $m,n \geq 1$, let $\mathcal{V}$ be a linear subspace of $M_{m,n}$, and let $$d = {\rm max}\{{\rm dim}(E) - {\rm dim}(\mathcal{V}[E]): E\mbox{ is a linear subspace of }\mathbb{F}^n\}.$$ Then $$\frac{1}{r}\cdot {\rm rank}(\mathcal{V}\otimes M_r) = n - d$$ for all $r \geq n-1$.
\end{thm}

I cannot find an explicit statement of this exact result in the literature, but it is obviously implied by Theorem 1.5/Remark 1.6 of \cite{IQS}, for example.

\subsection{Theorems of Hall,  Kőnig, and Dilworth}

Matrix versions of the theorems of Hall, Kőnig, and Dilworth are easy consequences of Theorem \ref{ncr}.

\begin{thm} (Matrix marriage theorem)
Let $m \geq n \geq 1$ and let $\mathcal{V}$ be a linear subspace of $M_{m,n}$. Then ${\rm dim}(\mathcal{V}[E]) \geq {\rm dim}(E)$ for every linear subspace $E$ of $\mathbb{F}^n$ if and only if there is a rank $rn$ matrix in $\mathcal{V} \otimes M_r$ for $r = n-1$. The same equivalence holds for any $r \geq n - 1$.
\end{thm}

\begin{proof}
Suppose there is a linear subspace $E$ of $\mathbb{F}^n$ with ${\rm dim}(\mathcal{V}[E]) < {\rm dim}(E)$. Then for any $r \geq 1$ $$(\mathcal{V}\otimes M_r)[E\otimes \mathbb{F}^r] = \mathcal{V}[E]\otimes \mathbb{F}^r$$ so $${\rm dim}((\mathcal{V}\otimes M_r)[E\otimes \mathbb{F}^r]) = r\cdot {\rm dim}(\mathcal{V}[E]) < r\cdot {\rm dim}(E) = {\rm dim}(E\otimes \mathbb{F}^r).$$ In particular, for any $A \in \mathcal{V}\otimes M_r$ we have $${\rm dim}(A(E\otimes \mathbb{F}^r)) < {\rm dim}(E\otimes \mathbb{F}^r),$$ showing that every matrix in $\mathcal{V}\otimes M_r$ has a nonzero kernel. So $\mathcal{V}\otimes M_r$ cannot contain any matrix of rank $rn$, for any $r \geq 1$.

Conversely, if ${\rm dim}(\mathcal{V}[E]) \geq {\rm dim}(E)$ for every $E$ then Theorem \ref{ncr} implies that there is a rank $rn$ matrix in $\mathcal{V}\otimes M_r$ for every $r \geq n-1$.
\end{proof}

In particular, for any linear subspace $\mathcal{V}$ of $M_n$, if there is an invertible matrix in $\mathcal{V}\otimes M_r$ for some $r \geq n - 1$ then there is an invertible matrix in $\mathcal{V}\otimes M_r$ for all $r \geq n - 1$, and this happens if and only if ${\rm dim}(\mathcal{V}[E]) \geq {\rm dim}(E)$ for every linear subspace $E$ of $\mathbb{F}^n$.

Let a {\it cover} for $\mathcal{V} \subseteq M_{m,n}$ be a pair $(E,F)$ of linear subspaces of $\mathbb{F}^n$ and $\mathbb{F}^m$ such that $\mathcal{V}[E^\perp] \subseteq F$. Equivalently, $Q\mathcal{V}P = \{0\}$ where $P$ and $Q$ are the orthogonal projections onto $E^\perp$ and $F^\perp$. Its {\it size} is the value ${\rm dim}(E) + {\rm dim}(F)$.

\begin{thm}\label{mkt} (Matrix Kőnig's theorem)
Let $m,n \geq 1$ and let $\mathcal{V}$ be a linear subspace of $M_{m,n}$. Then, for any $r \geq n - 1$, the maximum rank of a matrix in $\mathcal{V} \otimes M_r$ is $r$ times the minimum size of a cover for $\mathcal{V}$.
\end{thm}

\begin{proof}
For any $r \geq 1$, if $(E,F)$ is a cover for $\mathcal{V}$ then $(E\otimes \mathbb{F}^r, F\otimes \mathbb{F}^r)$ is a cover for $\mathcal{V}\otimes M_r$. That is, every $A \in \mathcal{V}\otimes M_r$ satisfies $$A((E\otimes \mathbb{F}^r)^\perp) = A(E^\perp\otimes \mathbb{F}^r) \subseteq F\otimes \mathbb{F}^r,$$ so that $${\rm rank}(A) \leq {\rm dim}(E\otimes \mathbb{F}^r) + {\rm dim}(F\otimes \mathbb{F}^r) = r\cdot ({\rm dim}(E) + {\rm dim}(F)).$$ Thus every matrix in $\mathcal{V} \otimes M_r$ has rank at most $r$ times the minimum size of a cover for $\mathcal{V}$, for any $r \geq 1$.

Conversely, let $(E_0,F_0)$ be a cover of minimum size and let $d_0 = n - ({\rm dim}(E_0) + {\rm dim}(F_0))$, so that $n - d_0$ is the minimum size of a cover. I claim that ${\rm dim}(\mathcal{V}[E]) \geq {\rm dim}(E) - d_0$ for every linear subspace $E$ of $\mathbb{F}^n$. If this were to fail for some $E$, then $(E^\perp, \mathcal{V}[E])$ would be a cover of size strictly less than $$(n - {\rm dim}(E)) + ({\rm dim}(E) - d_0) = n - d_0,$$ a contradiction. This proves the claim, and it follows that $$d_0 \geq d = \max\{{\rm dim}(E) - {\rm dim}(\mathcal{V}[E]): E\mbox{ is a linear subspace of }\mathbb{F}^n\},$$ so that $n - d_0$, the minimum size of a cover, is less than or equal to $n - d$. It now follows from Theorem \ref{ncr} that for any $r \geq n - 1$ there is a matrix in $\mathcal{V} \otimes M_r$ whose rank is at least $r$ times the minimum size of a cover.
\end{proof}

Turning to Dilworth's theorem, I do not have any matrix version of Theorem \ref{qdt}. Theorem \ref{cqdt} readily adapts to the matrix setting, however. 

A linear subspace $\mathcal{V} \subseteq M_n$ is a {\it nilpotent algebra} if it satisfies $\mathcal{V}^2 \subseteq \mathcal{V}$ and $\mathcal{V}^n = \{0\}$. We use the same definition of coherent decompositions into chains as before: a {\it coherent chain in $\mathbb{F}^n$ relative to $\mathcal{V}$} is a sequence $(v, Av, \ldots, A^{k-1}v)$ for some $v \in \mathbb{F}^n$ and some $A \in \mathcal{V}$, and a {\it coherent decomposition of $\mathbb{F}^n$ into chains relative to $\mathcal{V}$} is a collection of chains which are all implemented by one single $A \in \mathcal{V}$ and whose elements collectively form a basis of $\mathbb{F}^n$. Its {\it size} is the number of coherent chains in the family. Our definition of an {\it antichain} $C$ in Section 1.3 was made relative to a relation $R$; in terms of $\mathcal{V}$, the appropriate condition is that $C$ should be a linear subspace of $\mathbb{F}^n$ such that $\mathcal{V}[C] \subseteq C^\perp$. Equivalently, $P\mathcal{V}P = \{0\}$ where $P$ is the orthogonal projection onto $C$.

\begin{lem}\label{dillem2}
Let $n \geq 1$ and let $\mathcal{V} \subseteq M_n$ be a nilpotent algebra. Then the maximum dimension of an antichain for $\mathcal{V}$ equals $n$ minus the minimum size of a cover for $\mathcal{V}$.
\end{lem}

\begin{proof}
The proof is essentially the same as the proof of Lemma \ref{dillem}, rewritten in terms of $\mathcal{V}$ rather than $R$. The maximum dimension of an antichain is at least $n$ minus the minimum size of a cover by exactly the same reasoning as before. Conversely, given any antichain $C$, let $E = (C + \mathcal{V}[C])^\perp$ and $F = \mathcal{V}[C]$. Since $C$ is an antichain, $\mathcal{V}[C] \subseteq C^\perp$ and so ${\rm dim}(C + \mathcal{V}[C]) = {\rm dim}(C) + {\rm dim}(\mathcal{V}[C])$. Therefore \begin{eqnarray*}
{\rm dim}(E) + {\rm dim}(F) &=& (n - {\rm dim}(C) - {\rm dim}(\mathcal{V}[C])) + {\rm dim}(\mathcal{V}[C])\\ &=& n - {\rm dim}(C).
\end{eqnarray*} And using the fact that $\mathcal{V}$ is an algebra, we have $\mathcal{V}[\mathcal{V}[C]] \subseteq \mathcal{V}[C]$ and hence $$\mathcal{V}[E^\perp] = \mathcal{V}(C + \mathcal{V}[C]) = \mathcal{V}[C] + \mathcal{V}[\mathcal{V}[C]] = F,$$ showing that $(E,F)$ is a cover. Thus the maximum dimension of an antichain exactly equals $n$ minus the size of some cover.
\end{proof}

\begin{thm} (Matrix Dilworth's theorem)
Let $n \geq 1$ and let $\mathcal{V} \subseteq M_n$ be a nilpotent algebra. Then for any $r \geq n - 1$, the minimum size of a coherent decomposition of $\mathbb{F}^{rn}$ into chains relative to $\mathcal{V} \otimes M_r$ equals $r$ times the maximum dimension of an antichain for $\mathcal{V}$.
\end{thm}

\begin{proof}
Fix $r \geq n - 1$. It follows from Theorem \ref{mkt} plus Lemma \ref{dillem2} that the maximum rank of a matrix in $\mathcal{V}\otimes M_r$ is $r(n - c)$ where $c$ is the maximum dimension of an antichain for $\mathcal{V}$. We must therefore show that the minimum size of a coherent decomposition of $\mathbb{F}^{rn}$ into chains relative to $\mathcal{V} \otimes M_r$ equals $rn$ minus the maximum rank $s$ of a matrix in $\mathcal{V}\otimes M_r$. But this equality holds by exactly the same argument used to prove the analogous statement in the proof of Theorem \ref{cqdt}.
\end{proof}

I previously gave a matrix version of Dilworth's theorem in \cite[Theorem 6.8]{weaver}. However, that result was unsatisfying because it was only an inequality (minimum size of a partition into chains $\leq$ maximum size of an antichain). This is essentially the same issue we discussed in Section 1.3 where two $w$-chains spanned the entire vector space despite the existence of an antichain of size 3.

\subsection{Menger's theorem}

Let $\mathcal{V}$ be a linear subspace of $M_n$ and let $E$ and $F$ be linear subspaces of $\mathbb{F}^n$. Then an {\it $(E,F)$-separator} (relative to $\mathcal{V}$) is a pair of linear subspaces $\tilde{E}$ and $\tilde{F}$ of $\mathbb{F}^n$ with the properties that (1) $E \subseteq \tilde{E}$ and $F \subseteq \tilde{F}$, (2) $\tilde{F}^\perp \subseteq \tilde{E}$, and (3) $\mathcal{V}[\tilde{F}^\perp] \subseteq  \tilde{E}$. The {\it size} of the separator is ${\rm dim}(\tilde{E} \cap \tilde{F})$.

Let $\iota: E \to \mathbb{F}^n$ be the natural inclusion and let $\pi: \mathbb{F}^n \to F$ be the orthogonal projection. We define the {\it matricial path capacity between $E$ and $F$ relative to $\mathcal{V}$} to be $$MPC_{\mathcal{V}}(E,F) = \ncrank(\mathcal{W}) - n$$ where $\mathcal{W}$ is the space of linear maps from $\mathbb{F}^n \oplus E$ to $\mathbb{F}^n \oplus F$ spanned by the single matrix $\left[\begin{matrix}I&\iota\\ \pi&0\end{matrix}\right]$ and the matrices $\left[\begin{matrix}A&0\\0&0\end{matrix}\right]$ for $A \in \mathcal{V}$.

\begin{thm}\label{mmt} (Matrix Menger's theorem)
Let $n \geq 1$, let $\mathcal{V}$ be a linear subspace of $M_n$, and let $E$ and $F$ be linear subspaces of $\mathbb{F}^n$. Then the matricial path capacity between $E$ and $F$ relative to $\mathcal{V}$ equals the minimum size of an $(E,F)$-separator.
\end{thm}

\begin{proof}
Let $\mathcal{W}$ be defined as above. According to Theorem \ref{mkt}, $\ncrank(\mathcal{W})$ equals the minimum size of a cover for $\mathcal{W}$. That is, it is the minimum value of ${\rm dim}(C^\perp) + {\rm dim}(D)$ over all $C \subseteq \mathbb{F}^n \oplus E$ and $D \subseteq \mathbb{F}^n \oplus F$ satisfying $\mathcal{W}[C] \subseteq D$. Since $${\rm dim}(C^\perp) = {\rm dim}(\mathbb{F}^n \oplus E) - {\rm dim}(C) = n + {\rm dim}(E) - {\rm dim}(C),$$ we get that the matricial path capacity between $E$ and $F$ equals the minimum value of $${\rm dim}(E) - {\rm dim}(C) + {\rm dim}(D)$$ over all $C \subseteq \mathbb{F}^n \oplus E$ and $D \subseteq \mathbb{F}^n \oplus F$ satisfying $\mathcal{W}[C] \subseteq D$.

Let $(\tilde{E}, \tilde{F})$ be any $(E,F)$-separator. Then $C = \tilde{F}^\perp \oplus E$ and $D = \tilde{E} \oplus \{0\}$ satisfy $\mathcal{W}[C] \subseteq D$, so by the last paragraph
\begin{eqnarray*}
MPC_{\mathcal{V}}(E,F) &\leq& {\rm dim}(E) - {\rm dim}(C) + {\rm dim}(D)\\
&=&{\rm dim}(E) - (n - {\rm dim}(\tilde{F}) + {\rm dim}(E)) + {\rm dim}(\tilde{E})\\
&=& {\rm dim}(\tilde{E}) + {\rm dim}(\tilde{F}) - n\\
&=& {\rm dim}(\tilde{E} \cap \tilde{F}).
\end{eqnarray*}
Since $\tilde{E}$ and $\tilde{F}$ were arbitrary, this shows that the matricial path capacity between $E$ and $F$ is less than or equal to the minimum size of an $(E,F)$-separator.

Conversely, choose $C \subseteq \mathbb{F}^n \oplus E$ and $D \subseteq \mathbb{F}^n \oplus F$ that satisfy $\mathcal{W}[C] \subseteq D$. We normalize them in two steps, both of which either reduce or leave unaffected the quantity ${\rm dim}(E) - {\rm dim}(C) + {\rm dim}(D)$. First, define $C' = C + (\{0\} \oplus E)$ and $D' = D + (E \oplus \{0\})$. We still have $\mathcal{W}[C'] \subseteq D'$, but the dimensions change: $${\rm dim}(C') = {\rm dim}(C) + {\rm dim}(E) - {\rm dim}(C \cap (\{0\} \oplus E))$$ and $${\rm dim}(D') = {\rm dim}(D) + {\rm dim}(E) - {\rm dim}(D \cap (E \oplus \{0\})).$$ But the map $\left[\begin{matrix}I&\iota\\ \pi&0\end{matrix}\right]$ takes $C \cap (\{0\} \oplus E)$ isomorphically (i.e., without loss of dimension) into $D \cap (E \oplus \{0\})$, so the increase to ${\rm dim}(C)$ is at least as great as the increase to ${\rm dim}(D)$; that is, $${\rm dim}(E) - {\rm dim}(C) + {\rm dim}(D) \geq {\rm dim}(E) - {\rm dim}(C') + {\rm dim}(D').$$ 

Next, write $C' = C_0 \oplus E$ and define $C'' = C' \cap (F^\perp \oplus E) = (C_0 \cap F^\perp) \oplus E$ and $D'' = D' \cap (\mathbb{F}^n \oplus \{0\})$. The condition $\mathcal{W}[C''] \subseteq D''$ is still preserved, and the $C$ dimension now drops by $${\rm dim}(C') - {\rm dim}(C'') = {\rm dim}(C_0) - {\rm dim}(C_0 \cap F^\perp) = {\rm dim}(\pi(C_0)).$$ But for every $v \in C_0$ we have $$\left[\begin{matrix}v\\ \pi(v)\end{matrix}\right] = \left[\begin{matrix}I&\iota\\ \pi&0\end{matrix}\right]\left[\begin{matrix}v\\ 0\end{matrix}\right] \in D',$$ which shows that intersecting $D'$ with $\mathbb{F}^n \oplus \{0\}$ lowers its dimension by at least ${\rm dim}(\pi(C_0))$. Thus again \begin{eqnarray*}{\rm dim}(E) - {\rm dim}(C') + {\rm dim}(D') &\geq& {\rm dim}(E) - {\rm dim}(C'') + {\rm dim}(D'')\\
&=& {\rm dim}(D_1) - {\rm dim}(C_1)\end{eqnarray*} where $C_1 = C_0 \cap F^\perp$ and $D'' = D_1 \oplus \{0\}$. We conclude that this quantity that determines the matricial path capacity is minimized for $C$ and $D$ of the form $C'' = C_1 \oplus E$ and $D'' = D_1 \oplus \{0\}$.

Finally, define $\tilde{E} = D_1$ and $\tilde{F} = C_1^\perp$. The axioms for an $(E,F)$-separator are straightforwardly verified, and we have $${\rm dim}(\tilde{E} \cap \tilde{F}) = {\rm dim}(\tilde{E}) - {\rm dim}(\tilde{F}^\perp) = {\rm dim}(D_1) - {\rm dim}(C_1).$$ So the minimum size of an $(E,F)$-separator is less than or equal to the matricial path capacity between $E$ and $F$. We conclude that the two quantities are equal.
\end{proof}

The proof strategy for Theorem \ref{mmt} is a little different from the one used for Theorem \ref{qmt}. There we had the advantage of being able to describe $\mathcal{V} = \mathcal{V}_R$ using a finite set of pairs of vectors, but the disadvantage of not being able to directly invoke the linear Kőnig's theorem (Theorem \ref{qkt}) because the $\mathcal{W}$ space corresponding to $\mathcal{V}_R$ is not rank one generated.

\section*{Acknowledgement}

Much of the work on which this paper was based was done with the assistance of an AI model (Google Gemini).

\bibliographystyle{amsplain}

\end{document}